\documentclass{amsart}
\usepackage{amsfonts,amssymb,amsmath,amsthm}
\usepackage{url}
\usepackage{enumerate}
\usepackage{epsfig}
\usepackage{tikz}
\usepackage{ulem}
\usetikzlibrary{shapes,arrows}

\newtheorem{theorem}{Theorem}[section]
\newtheorem{lemma}[theorem]{Lemma}
\newtheorem{corollary}[theorem]{Corollary}
\newtheorem{proposition}[theorem]{Proposition}

\theoremstyle{definition}
\newtheorem{definition}[theorem]{Definition}
\newtheorem{example}[theorem]{Example}

\theoremstyle{remark}
\newtheorem{remark}[theorem]{Remark}

\numberwithin{equation}{section}

\tikzstyle{block} = [rectangle, draw, text width=5.3em, text centered, rounded corners, minimum height=4em]
\tikzstyle{blocknob} = [text width=4.4em, text centered, minimum height=4em]
\tikzstyle{line} = [draw, -latex']

\def\W {\mathrm{W}}
\def\D {\mathcal{D}}
\def\A {\mathcal{A}}
\def\L {\mathcal{L}}

\begin{document}

\title{Nevanlinna Theory of the Wilson Divided-difference Operator}


\author[K. H. CHENG]{Kam Hang CHENG}
\address{Department of Mathematics, The Hong Kong University of Science and Technology, Clear Water Bay, Kowloon, Hong Kong.}
\email{henry.cheng@family.ust.hk}

\author[Y. M. CHIANG]{Yik-Man CHIANG}
\address{Department of Mathematics, The Hong Kong University of Science and Technology, Clear Water Bay, Kowloon, Hong Kong.}
\email{machiang@ust.hk}
\thanks{Both authors were partially supported by GRF no. 16306315 from Hong Kong Research Grant Council.\ \ The second author was also partially supported by GRF no. 600609.}

\subjclass[2010]{Primary 30D35; Secondary 30D30, 33C45, 39A05}

\date{September 14, 2016.\ \ To appear in Annales Academi\ae\ Scientiarum Fennic\ae\ Mathematica.}

\dedicatory{Dedicated to the fond memories of Rauno Aulaskari}

\keywords{Wilson divided-difference operator, complex function theory, Poisson-Jensen formula, Nevanlinna theory}

\begin{abstract}
Sitting at the top level of the Askey-scheme, Wilson polynomials are regarded as the most general hypergeometric orthogonal polynomials.\ \ Instead of a differential equation, they satisfy a second order Sturm-Liouville type difference equation in terms of the Wilson divided-difference operator.\ \ This suggests that in order to better understand the distinctive properties of Wilson polynomials and related topics, one should use a function theory that is more natural with respect to the Wilson operator.\ \ Inspired by the recent work of Halburd and Korhonen, we establish a full-fledged Nevanlinna theory of the Wilson operator for meromorphic functions of finite order.\ \ In particular, we prove a Wilson analogue of the lemma on logarithmic derivatives, which helps us to derive Wilson operator versions of Nevanlinna's Second Fundamental Theorem, some defect relations and Picard's Theorem.\ \ These allow us to gain new insights on the distributions of zeros and poles of functions related to the Wilson operator, which is different from the classical viewpoint.\ \ We have also obtained a relevant five-value theorem and Clunie type theorem as applications of our theory, as well as a pointwise estimate of the logarithmic Wilson difference, which yields new estimates to the growth of meromorphic solutions to some Wilson difference equations and Wilson interpolation equations.
\end{abstract}

\maketitle
\setcounter{tocdepth}{1}
\tableofcontents

\section{Introduction}

The Wilson divided-difference operator $\D_\W$ was first considered by Wilson to study \textit{Wilson polynomials} $W_n(x;a,b,c,d)$ \cite[p. 34]{Askey-Wilson}, defined by
\begin{align*}
	&\ \ \ \,\frac{W_n(x;\, a,\, b,\, c,\, d)}{(a+b)_n(a+c)_n(a+d)_n} \\
	&:={}_4F_3\left(
	\begin{matrix}
	\begin{matrix}
		-n, & n+a+b+c+d-1, & a+i\sqrt{x}, & a-i\sqrt{x}
	\end{matrix} \\
	\begin{matrix}
		a+b, & a+c, & a+d
	\end{matrix}
	\end{matrix}
	\ ;\ 1\right),
\end{align*}
which are hypergeometric orthogonal polynomials located at the top level of the Askey scheme \cite{Koekoek-Swarttouw}.\ \ They are the most general hypergeometric orthogonal polynomials that contain all the known classical hypergeometric orthogonal polynomials \cite{Andrews-Askey} as special cases.\ \ The Wilson operator acts on Wilson polynomials in a similar manner as the usual differential operator acts on monomials, except with a shift in the parameters $a$, $b$, $c$ and $d$:
\begin{align*}
	(\D_\W W_n)(x;a,b,c,d)=C_n W_{n-1}\Big(x;a+\frac{1}{2},b+\frac{1}{2},c+\frac{1}{2},d+\frac{1}{2}\Big),
\end{align*}
where $C_n=-n(n+a+b+c+d-1)$.\ \ Instead of a differential equation, it is known that Wilson polynomials satisfy a second order Sturm-Liouville type difference equation in terms of the Wilson operator (see \S\ref{sec:Applications}).

Wilson polynomials are intimately related to other classical orthogonal polynomials on which active research have been done.\ \ Wilson identified himself in \cite{Wilson} that an orthogonality relation for the 6-$j$ symbols in the coupling of angular momenta in quantum mechanics follow from that of Wilson polynomials as a special case.\ \ Koornwinder \cite{Koornwinder_1985} discovered that Wilson polynomials and Jacobi polynomials are mapped to each other via the Jacobi-Fourier transform, and studied Wilson polynomials from a group theoretic viewpoint \cite{Koornwinder_1986}.\ \ Groenevelt \cite{Groenevelt_2003} used Wilson functions, which is a family of transcendental solutions of the aforementioned second-order Sturm-Liouville Wilson difference equation linearly independent to Wilson polynomials, as the kernel of a new integral transform called the Wilson function transform.\ \ Wilson polynomials have also been extended to a multiple-parameter version \cite{BCA_2005} using essentially the same tool of Jacobi transform that was used in \cite{Groenevelt_2003}.\ \ Wilson polynomials also have applications in various aspects from theoretical physics to birth and death processes, see for examples \cite{Bender-Meisinger-Wang}, \cite{ILVW_1990}, \cite{KMP_2007}, \cite{Masson_1991} and \cite{Mimachi_1999}.

It is clear that in a majority of the research works cited, the subject of investigation was either Wilson polynomials or Wilson functions.\ \ In this paper we would therefore like to look into the Wilson operator in a broader function theoretic context, and establish some results about its interaction with meromorphic functions.\ \ It turns out that Halburd and Korhonen have recently established a Nevanlinna theory in \cite{Halburd-Korhonen} with respect to the ordinary difference operator $(\Delta f)(x)=f(x+\eta)-f(x)$ for each fixed non-zero $\eta$.\ \ In particular, a new difference type little Picard theorem has been proved.\ \ See also \cite{HK-3}.\ \ This suggests that the function theory with respect to $\Delta$ is somewhat different from the classical function theory with respect to $\frac{d}{dx}$, despite all the functions considered in both cases are meromorphic.\ \ Indeed, N\"orlund spent a large part of his 1926 memoir \cite{Noerlund}, which has not been mentioned by most of the recent researchers on orthogonal polynomials, on investigating how to expand entire functions into a Stirling interpolation series.\ \ Stirling series converge faster than Newton series, and the interpolating polynomials in Stirling series are obtained by slightly modifying the Wilson polynomials.\ \ Thus it is natural to develop a function theory using Wilson polynomials $\{W_n\}$ (or simply $\{\tau_n(\cdot;a)\}$ where $\tau_n(x;a) := \prod_{k=0}^{n-1}{[(a+ki)^2-x]}$) as a basis, instead of the usual $\{(x-a)^n\}$.\ \ In fact, following the classical idea as in \cite{Noerlund}, one can show that every entire function $f$ satisfying the growth condition
\[
	\limsup_{r\to\infty}{\frac{\ln{M(r,f)}}{\sqrt{r}}} < 2\ln 2
\]
admits, for each $a\in\mathbb{C}$, a \textit{Wilson series expansion}
\[
	\sum_{k=0}^\infty{a_k\tau_k(x;a)}
\]
which converges uniformly to $f$ on any compact subset of $\mathbb{C}$.\ \ As a result of this series expansion, the appropriate way of counting zeros (and poles) with respect to $\D_\W$ is different compared to that with respect to $\frac{d}{dx}$.\ \ This is reflected in the way we define the \textit{Wilson counting functions} in \S\ref{sec:Proofs} (see Definition \ref{WilsonCount}).\ \ The second author and Feng have also established a full-fledged Nevanlinna theory with respect to the Askey-Wilson divided-difference operator $\D_q$ earlier in \cite{Chiang-Feng3}.\ \ Parallel to what has been obtained in the MPhil thesis of the first author in 2013, we demonstrate in this paper that the Wilson operator $\D_\W$ has its own version of Nevanlinna theory, and establish a number of key results in this version, including defect relations and a Wilson version of little Picard theorem.\ \ A \textit{Wilson exceptional value} of a meromorphic function is a value in $\hat{\mathbb{C}}$ whose preimage lies on a certain special sequence (see Definition~\ref{ExVWilson}).\ \ A transcendental meromorphic function $f\notin\ker\D_\W$ of finite order can possess at most two Wilson exceptional values, because such values $a$ have defect $\Theta_\W(a,f)=1$ (see Remark~\ref{ExVWilsonRmk}) while our defect relations (Corollary~\ref{Defect}) assert that
\[
	{\sum_{a\in\hat{\mathbb{C}}}\Theta_\W(a,f)\le 2}.
\]
As an example (Example~\ref{defectegs} (i)), the entire function
\[
	f(x):=\prod_{k=0}^{\infty}{\left[1-\frac{x}{(b+ki)^2}\right]}
\]
has simple zeros only at the points
\[
	x_k=(b+ki)^2 \hspace{25px} \mbox{for }k=0,1,2,\ldots,
\]
so by Definition~\ref{ExVWilson}, $0$ is a Wilson exceptional value of $f$ and $\Theta_\W(0,f)=1$, which means that the value 0 is missed with respect to the Wilson operator.\ \ Note that the value $\infty$ is also missed in the classical sense as $f$ is entire, and so we clearly have $\Theta_\W(\infty,f)=1$ as well.\ \ As another example (Example~\ref{defectegs} (ii)), the entire function
\begin{equation*}
	\varphi(x) = \,_2F_1\left(
		\begin{matrix}
		\begin{matrix}
			a+i\sqrt{x}, & \frac12+i\sqrt{x}
		\end{matrix} \\
			a+\frac12
		\end{matrix}
	\ ;\ -1
	\right)\,_2F_1\left(
		\begin{matrix}
		\begin{matrix}
			a-i\sqrt{x}, & \frac12-i\sqrt{x}
		\end{matrix} \\
			a-\frac12
		\end{matrix}
	;\ -1
	\right)
\end{equation*}
which is a special case of the generating function $\varphi(x;t)$ of Wilson polynomials, has double zeros only at the points
\[
	x_k=-(a-1+2k)^2 \hspace{25px} \mbox{for }k=1,2,3,\ldots,
\]
so $0$ is a ``$2i$-shift" Wilson exceptional value of $\varphi$, and $\Theta_{\W,2i}(0,\varphi)=1$.\ \ Since $\varphi$ is entire, we similarly have $\Theta_\W(\infty,f)=1$.\ \ The corresponding residue calculus with respect to $\D_\W$ and the relation between this operator and interpolation theory will be discussed in a subsequent paper.

This paper is organized as follows.\ \ In \S\ref{sec:DW}, we will first give the definition and some basic properties of $\D_\W$.\ \ In the subsequent sections, we will then follow the classical approach on the Nevanlinna theory of the differential operator.\ \ We will state in \S\ref{sec:Mainresults} our main results, which include a Wilson operator version of Nevanlinna's Second Fundamental Theorem and some corollaries about defect relations.\ \ Before proving all these in \S\ref{sec:Proofs}, we will introduce a Wilson operator analogue of the lemma on logarithmic derivative in \S\ref{sec:WilsonLogLemma}.\ \ Some results from Halburd and Korhonen \cite{Halburd-Korhonen} and from the second author and Feng's work on the Nevanlinna theories of the ordinary difference operator \cite{Chiang-Feng1} and of the Askey-Wilson operator \cite{Chiang-Feng3} will be useful to our proofs.\ \ In \S\ref{sec:Ptwise}, we will give a \textit{pointwise} estimate for the logarithmic Wilson difference, as opposed to the \textit{overall} estimate given in \S\ref{sec:WilsonLogLemma}.\ \ These estimates will be applied to give estimates on the growth of meromorphic solutions to some specific types of Wilson difference equations and interpolation equations in \S\ref{sec:Applications} and \S\ref{sec:Ptwise}.

In this paper, we adopt the following notations:
\begin{enumerate}[(i)]
	\item $\mathbb{N}$ denotes the set of all natural numbers \textit{excluding} $0$, and $\mathbb{N}_0:=\mathbb{N}\cup\{0\}$.
	\item For every positive real number $r$ and every complex number $a$, $D(a;r)$ denotes the open disk of radius $r$ centered at $a$ in the complex plane.
	\item For every positive real number $r$, $\ln{r}$ denotes the natural logarithm of $r$, while $\ln^+{r}$ denotes the non-negative number $\max{\{\ln{r},0\}}$.
	\item A \textit{complex function} always means a function in \textit{one} complex variable, and a \textit{meromorphic function} always means a meromorphic function from $\mathbb{C}$ to $\hat{\mathbb{C}}:=\mathbb{C}\cup\{\infty\}$, unless otherwise specified.
	\item A summation notation of the form $\displaystyle\sum_{|a_\nu|<r}$ denotes a sum running over all the $\nu$'s such that the term $a_\nu$ of the sequence $\{a_\nu\}_\nu$ has modulus smaller than $r$.
	\item For any two functions $f,g:[0,\infty)\to\mathbb{R}$, we write
	\begin{itemize}
		\item $g(r)=O(f(r))$ as $r\to\infty$ if and only if there exist $C>0$ and $M>0$ such that $|g(r)|\le C|f(r)|$ whenever $r>M$;
		\item $g(r)=o(f(r))$ as $r\to\infty$ if and only if for every $C>0$, there exists $M>0$ such that $|g(r)|\le C|f(r)|$ whenever $r>M$;
		\item $f(r)\sim g(r)$ as $r\to\infty$ if and only if for every $\varepsilon>0$, there exists $M>0$ such that $\left|\frac{f(r)}{g(r)}-1\right|<\varepsilon$ whenever $r>M$.
	\end{itemize}
\end{enumerate}

\section{The Wilson Operator}
\label{sec:DW}

In this section, we give the definition of the Wilson operator and a few of its properties.

\begin{definition}
	Let $\sqrt{\cdot}$ be a branch of the complex square-root with the imaginary axis as the branch cut.\ \ For each $x\in\mathbb{C}$ we denote
	\[
		x^+ := \left(\sqrt{x}+\frac{i}{2}\right)^2 \hspace{30px} \mbox{and} \hspace{30px} x^- := \left(\sqrt{x}-\frac{i}{2}\right)^2.
	\]
	We also adopt the notations $x^{\pm(0)}:=x$, $x^{\pm(m)}:=(x^{\pm(m-1)})^\pm$ and $x^{\pm(-m)}:=x^{\mp(m)}$ for every positive integer $m$.\ \ Then we define the \textbf{\textit{Wilson operator}} $\D_\W$, which acts on all complex functions, as follows:
	\begin{equation}
		\label{A1} (\D_\W f)(x) := \frac{f(x^+)-f(x^-)}{x^+ - x^-} = \frac{f((\sqrt{x}+\frac{i}{2})^2)-f((\sqrt{x}-\frac{i}{2})^2)}{2i\sqrt{x}}.
	\end{equation}
	To simplify notations, we also define the \textbf{\textit{Wilson averaging operator}} $\A_\W$ by
	\begin{equation}
		\label{A2} (\A_\W f)(x) := \frac{f(x^+)+f(x^-)}{2} = \frac{f((\sqrt{x}+\frac{i}{2})^2) + f((\sqrt{x}-\frac{i}{2})^2)}{2}.
	\end{equation}
\end{definition}

Although there are two choices of $\sqrt{x}$ for each $x\ne 0$, $\D_\W$ and $\A_\W$ are independent of the choice of $\sqrt{\cdot}$ and are thus always well-defined.\ \ Moreover, the value of $\D_\W f$ at $0$ should be defined as
	\[
		(\D_\W f)(0) := \lim_{x\to 0}(\D_\W f)(x) = f'\left(-\frac{1}{4}\right),
	\]
in case $f$ is differentiable at $-\frac{1}{4}$.\ \ We sometimes write $z:=\sqrt{x}$.

We consider branches of square-root with the imaginary axis as the branch cut because of the shift of $\frac{i}{2}$ in the definition of $\D_\W$.\ \ One can consider Wilson operators of other shifts $\frac{c}{2}$ where $c\in\mathbb{C}\setminus\{0\}$, and accordingly choose branches with the line joining $c$ and $0$ as the branch cut, without affecting the later results. \\

$\D_\W$ is a linear operator, whose kernel contains precisely all those functions of the form $g\circ\sqrt{\cdot}$ such that $g$ is even, meromorphic and periodic with period $i$, where $\sqrt\cdot$ denotes a particular branch of the square-root function.\ \ Equivalently, this means that
\begin{align*}
	\ker{\D_\W}= \left\{g\circ\sqrt{\cdot}:
		\begin{matrix}
		g\in\mathcal{M}(\mathbb{C}),\mbox{ $g\left(\frac{ni}{2}+z\right)=g\left(\frac{ni}{2}-z\right)$ for all $z\in\mathbb{C}$ and $n\in\mathbb{Z}$,}\\
		\mbox{and $\sqrt{\cdot}$ is a branch of the square-root function}
		\end{matrix}
	\right\}.
\end{align*}
Examples of elements in this kernel include the functions $\cosh({2\pi\sqrt{x}})$ and $\wp(\omega i\sqrt{x})$, where $\wp$ is the Weierstrass' P-function and $\omega$ is one of its periods.\ \ It also worths noting that the evenness and meromorphicity of $g$ ensures that $g\circ\sqrt{\cdot}$ is also meromorphic. \\

The Wilson operator clearly has the following product rule and quotient rule.
\begin{proposition}
\textup{(Wilson product and quotient rules)}
\label{PQRule}
	For every pair of complex functions $f_0$ and $f_1$, we have
	\[
		(\D_\W(f_0f_1))(x) = (\A_\W f_0)(x)(\D_\W f_1)(x) + (\A_\W f_1)(x)(\D_\W f_0)(x).
	\]
	If we assume, in addition, that $f_0\not\equiv 0$, then we also have
	\[
		\left(\D_\W\left(\frac{f_1}{f_0}\right)\right)(x) = \frac{(\A_\W f_0)(x) (\D_\W f_1)(x) - (\A_\W f_1)(x) (\D_\W f_0)(x)}{f_0(x^+)f_0(x^-)}.
	\]
\end{proposition}

\bigskip
We can easily check from the definition of the Wilson operator $\D_\W$ that it sends polynomials to polynomials.\ \ In fact, we have the following.
\begin{proposition}
\label{Mero}
	Let $f$ be a complex function.\ \ Then
	\begin{enumerate}[(i)]
	\item if $f$ is entire, then $\D_\W f$ and $\A_\W f$ are also entire;
	\item if $f$ is meromorphic, then $\D_\W f$ is also meromorphic; and
	\item if $f$ is rational, then $\D_\W f$ is also rational.
	\end{enumerate}
\end{proposition}
\begin{proof}
	We first prove (i).\ \ If $f$ is entire, then it has a Maclaurin series expansion $\displaystyle f(x)=\sum_{k=0}^\infty{a_k x^k}$ satisfying $\displaystyle\limsup_{k\to\infty}|a_k|^{\frac{1}{k}}=0$.\ \ Now
	\begin{align*}
		(\D_\W f)(x) &= \frac{f(x^+)-f(x^-)}{2iz} = \frac{1}{2iz}\left[\sum_{k=0}^\infty{a_k\left(z+\frac{i}{2}\right)^{2k}} - \sum_{k=0}^\infty{a_k\left(z-\frac{i}{2}\right)^{2k}}\right] \\
		&= \frac{1}{2iz}\sum_{k=1}^\infty{\left[2a_k\sum_{l=0}^{k-1}{\binom{2k}{2l+1}}z^{2l+1}\left(\frac{i}{2}\right)^{2k-2l-1}\right]} \\
		&= \sum_{l=0}^\infty{\left[\sum_{k=l+1}^\infty{\binom{2k}{2l+1}\frac{(-1)^{k-l-1}a_k}{2^{2k-2l-1}}}\right]x^l} = \sum_{l=0}^\infty{b_l x^l}
	\end{align*}
	is a Maclaurin series expansion where $\displaystyle b_l:=\sum_{k=l+1}^\infty{\binom{2k}{2l+1}\frac{(-1)^{k-l-1}a_k}{2^{2k-2l-1}}}$ satisfies
	\[
		|b_l|^{\frac{1}{l}} \le 4\cdot\left(\sum_{k=l+1}^\infty{\binom{2k}{2l+1}\frac{|a_k|}{4^k}}\right)^{\frac{1}{l}} \le 4\cdot\left(\sum_{k=l+1}^\infty{|a_k|}\right)^{\frac{1}{l}},
	\]
	so $\displaystyle\limsup_{l\to\infty}|b_l|^{\frac{1}{l}}=0$. This means that the Maclaurin series expansion of $\D_\W f$ has infinite radius of convergence, so $\D_\W f$ is entire.\ \ By a similar computation, one can easily show that $\A_\W f$ is also entire. \\

	Next we prove (ii).\ \ If $f$ is meromorphic, then we can write $f=\frac{f_1}{f_0}$, where $f_0$ and $f_1$ are entire functions without common zeros and $f_0\not\equiv 0$.\ \ According to the proven statement (i) and the quotient rule in Proposition~\ref{PQRule}, it suffices to check that the product $f_0(x^+)f_0(x^-)$ is entire whenever $f_0(x)$ is entire.

	Now if $f_0$ is entire, then it has a Maclaurin series expansion $\displaystyle f_0(x)=\sum_{k=0}^{\infty}{c_k x^k}$ satisfying $\displaystyle\limsup_{k\to\infty}|c_k|^{\frac{1}{k}}=0$.\ \ So we have
	\begin{align*}
		f_0(x^+)f_0(x^-) = \sum_{k=0}^{\infty}{\sum_{j=0}^{\infty}{c_k c_j\left(z+\frac{i}{2}\right)^{2k}\left(z-\frac{i}{2}\right)^{2j}}}.
	\end{align*}
	To put this series into a Maclaurin series in $x$, it suffices to show that each symmetric term $(z+\frac{i}{2})^{2k}(z-\frac{i}{2})^{2j} + (z+\frac{i}{2})^{2j}(z-\frac{i}{2})^{2k}$ is a polynomial in $x$.\ \ Without loss of generality, we may assume that $j\ge k$ and write $n:=j-k\ge 0$.\ \ Then binomial expansion gives
	\begin{align*}
		&\ \ \ \,\left(z+\frac{i}{2}\right)^{2k}\left(z-\frac{i}{2}\right)^{2j} + \left(z+\frac{i}{2}\right)^{2j}\left(z-\frac{i}{2}\right)^{2k} \\
		&= \left(z+\frac{i}{2}\right)^{2k}\left(z-\frac{i}{2}\right)^{2k} \left[\left(z+\frac{i}{2}\right)^{2n} + \left(z-\frac{i}{2}\right)^{2n}\right]\\
		&= 2\left(x+\frac{1}{4}\right)^{2k} \sum_{l=0}^{n}{\frac{(-1)^{n-l}}{2^{2n-2l}}\binom{2n}{2l} x^{l}},
	\end{align*}
	which is a polynomial in $x$.\ \ The Maclaurin series of $f_0(x^+)f_0(x^-)$ thus obtained has infinite radius of convergence by a similar argument as that in (i). \\

	Finally, to prove (iii), it just suffices to repeat the proof of (ii) verbatim, with the holomorphic functions $f_0$ and $f_1$ replaced by polynomials, and the infinite Maclaurin series expansion of $f_0$ replaced by a finite sum.
\end{proof}

\section{Main Results}
\label{sec:Mainresults}
We first recall some basic notions in the classical Nevanlinna theory, which can be found in many texts, for instance \cite{Cherry-Ye}, \cite{Hayman} and \cite{Yang}.\ \ Given a meromorphic function $f\not\equiv0$ with sequence of poles $\{b_\mu\}_\mu$ repeated according to multiplicity, the proximity function $m(r,f)$, the integrated counting function $N(r,f)$ and the Nevanlinna characteristic function $T(r,f)$ for $f$ are defined, for every $r>0$, as
\begin{align*}
	m(r,f) &:= \frac{1}{2\pi}\int_{0}^{2\pi}{\ln^+ \left|f(re^{i\theta})\right| d\theta}, \\
	N(r,f) &:= \int_{0}^{r}{\frac{n(t,f)-n(0,f)}{t}\,dt} + n(0,f) \ln r, \\
	T(r,f) &:= m(r,f) + N(r,f)
\end{align*}
respectively, where $\displaystyle n(r,f) := \sum_{|b_\mu|\le r}{1}$ for every $r\ge 0$.\ \ Nevanlinna's First Fundamental Theorem then states that for every meromorphic function $f$ and every complex number $a$, we have
\begin{align}
	\label{B1} m\left(r,\frac{1}{f-a}\right) + N\left(r,\frac{1}{f-a}\right) \equiv T\left(r,\frac{1}{f-a}\right) = T(r,f)+O(1)
\end{align}
as $r\to\infty$. \\

The order of a meromorphic function $f$ is defined by
\[
	\sigma:=\limsup_{r\to\infty}{\frac{\ln^+ T(r,f)}{\ln r}},
\]
which is either a non-negative real number or $+\infty$.\ \ The following is a fundamental inequality leading to our main result.
\begin{theorem}
\label{FundIneq}
\textup{(The Fundamental inequality)}
	Let $f$ be a meromorphic function of finite order $\sigma$ such that $f\notin\ker\D_\W$, $q$ be a positive integer, and $y_1, y_2, \ldots, y_q$ be $q$ complex numbers.\ \ Then for every $\varepsilon>0$, we have
	\begin{align}
		\label{C1} m(r,f) + \sum_{n=1}^q{m\left(r,\frac{1}{f-y_n}\right)} \le 2T(r,f) - N_\W(r) + O(r^{\sigma-\frac{1}{2}+\varepsilon}) + O(1)
	\end{align}
	as $r\to\infty$, where the \textbf{Wilson ramification term} $N_\W(r)$ is defined as
	\begin{align}
		\label{C2} N_\W(r) := N\left(r,\frac{1}{\D_\W f}\right) + 2N(r,f) - N(r,\D_\W f).
	\end{align}
\end{theorem}

To better interpret the Wilson ramification term $N_\W(r)$ in Theorem~\ref{FundIneq}, we need to introduce new counting functions $n_\W(r,f)$, $N_\W(r,f)$, $\widetilde{n_\W}(r,f)$ and $\widetilde{N_\W}(r,f)$, which are more suitable for the Wilson operator.\ \ While the precise definitions of these Wilson counting functions shall be given in \S\ref{sec:Proofs} (Definition~\ref{WilsonCount}), our main result in this paper is the following Wilson operator version of Nevanlinna's Second Fundamental Theorem, which is stated in terms of these new counting functions.
\begin{theorem}
\label{NSFTWilson}
\textup{(NSFT for the Wilson operator)}
	Let $f$ be a meromorphic function of finite order $\sigma$ such that $f\notin\ker\D_\W$, $q$ be a positive integer, and $y_1, y_2, \ldots, y_q$ be $q$ complex numbers.\ \ Then for every $\varepsilon>0$, we have
	\begin{align}
		\label{C12} (q-1)T(r,f) \le \widetilde{N_\W}(r,f) + \sum_{n=1}^{q}{\widetilde{N_\W}\left(r,\frac{1}{f-y_n}\right)} + O(r^{\sigma-\frac{1}{2}+\varepsilon}) + O(\ln r)
	\end{align}
	as $r\to\infty$.
\end{theorem}
Note that while the classical version of Nevanlinna's Second Fundamental Theorem imposes no restriction on the order of the meromorphic function $f$, our Wilson version works essentially just for meromorphic functions whose order is finite.\ \ In Halburd and Korhonen's ordinary difference operator version \cite{Halburd-Korhonen}, the meromorphic function is also required to have finite order; and in the second author and Feng's Askey-Wilson operator version \cite{Chiang-Feng3}, the meromorphic function is required to have finite logarithmic order. \\

After we have introduced in \S\ref{sec:Proofs} (Definition~\ref{WTheta}) the Wilson analogues $\vartheta_\W(a,f)$ and $\Theta_\W(a,f)$ of the ramification index $\vartheta(a,f)$ of $f$ at $a$ and the quantity $\Theta(a,f)$ in the classical Nevanlinna theory, the following corollary is an immediate consequence of Theorem~\ref{NSFTWilson}.\ \ This is also analogous to Halburd and Korhonen's defect relation \cite[Corollary 2.6]{Halburd-Korhonen} with respect to the ordinary difference operator.
\begin{corollary}
\label{Defect}
\textup{(Wilson deficient values)}
	Let $f$ be a transcendental meromorphic function of finite order such that $f\notin\ker\D_\W$.\ \ Then $\Theta_\W(a,f)=0$ except for at most countably many $a\in\hat{\mathbb{C}}$, and
	\[
		\sum_{a\in\hat{\mathbb{C}}}{[\delta(a,f)+\vartheta_\W(a,f)]} \le \sum_{a\in\hat{\mathbb{C}}}{\Theta_\W(a,f)} \le 2.
	\]
\end{corollary}

A Wilson generalization of Picard's Theorem can be obtained from the defect relation given in Corollary~\ref{Defect}.\ \ It gives a sufficient condition for a transcendental meromorphic function of finite order to become a function in $\ker\D_\W$.

\begin{theorem}
\label{WilsonPicard}
\textup{(Picard's Theorem for the Wilson operator)}
	Let $f$ be a meromorphic function of finite order.\ \ If $f$ has three distinct Wilson exceptional values, then either $f\in\ker\D_\W$ or $f$ is rational.
\end{theorem}
The notion of Wilson exceptional values will be defined precisely in \S\ref{sec:Proofs} (Definition~\ref{ExVWilson}).

\section{Lemma on Logarithmic Wilson Difference}
\label{sec:WilsonLogLemma}

The classical Nevanlinna theory starts from the \textit{Poisson-Jensen formula}, which states a relationship between the modulus of a function and the distribution of its poles and zeros.\ \ From this formula, we can obtain the lemma on logarithmic derivative, which is about the growth of $m\big(r,\frac{f'}{f}\big)$, i.e. the effect of the differential operator on the proximity function.

Now we aim to analyze the effect of the Wilson operator on the proximity function, so we start from the Poisson-Jensen Formula and develop a Wilson operator analogue of the lemma on logarithmic derivative.\ \ The following is our desired result, which will be a crucial step to the proof of our main results mentioned in \S\ref{sec:Mainresults}. This estimate can be compared with that for the usual difference operator $\Delta$ and meromorphic functions of finite order obtained by the second author and Feng \cite{Chiang-Feng1} and independently by Halburd and Korhonen \cite{HK-1}, as well as that for the Askey-Wilson operator $\D_q$ and meromorphic functions of finite logarithmic order also obtained by the second author and Feng \cite{Chiang-Feng3}.

\begin{theorem}
\label{WilsonLogLemma}
\textup{(Lemma on logarithmic Wilson difference)}
	If $f\not\equiv0$ is a meromorphic function of finite order $\sigma$, then for every $\varepsilon>0$,
	\[
		m\left(r,\frac{\D_\W f}{f}\right) = O(r^{\sigma-\frac{1}{2}+\varepsilon})
	\]
	as $r\to\infty$.
\end{theorem}
For the simple example $f(x):=e^x$ we have $\sigma=1$ and $m\left(r,\frac{\D_\W f}{f}\right) \sim \frac{2}{\pi}\sqrt r$ as $r\to\infty$, which shows that the number $\frac12$ on the right-hand side of Theorem~\ref{WilsonLogLemma} is the best possible. \\

We need the following lemmas in the course of the proof of Theorem~\ref{WilsonLogLemma}.

\begin{lemma}
\label{Estimate1}
	Suppose that $R>\frac{1}{4}$ and $0\le r<(\sqrt{R}-\frac{1}{2})^2$.\ \ Then the followings hold for every $z\in\partial D(0;\sqrt{r})$:
	\begin{enumerate}[(i)]
		\item For every $\phi\in[0,2\pi]$,
		\[
			\mathfrak{R}\left[\frac{Re^{i\phi}(2iz-\frac{1}{2})}{(Re^{i\phi}-z^2)(Re^{i\phi}-(z+\frac{i}{2})^2)}\right] \le \frac{R(2\sqrt{r}+\frac{1}{2})}{(R-r)(R-(\sqrt{r}+\frac{1}{2})^2)}.
		\]
		\item For every $w\in D(0;r)$,
		\[
			\left|\ln\left|\frac{R^2-\overline{w}(z+\frac{i}{2})^2}{R^2-\overline{w}z^2}\right|\right| \le \frac{2\sqrt{r}+\frac{1}{2}}{R-(\sqrt{r}+\frac{1}{2})^2}.
		\]
		\item For each $0<\alpha\le1$, there exists a constant $C_\alpha>0$ such that
			\begin{align*}
				\left|\ln\left|\frac{(z+\frac{i}{2})^2-w}{z^2-w}\right|\right| \le C_\alpha\left(\sqrt{r}+\frac{1}{4}\right)^\alpha\left[\frac{1}{|z^2-w|^\alpha} + \frac{1}{|(z+\frac{i}{2})^2-w|^\alpha}\right]
			\end{align*}
		for every $w\in D(0;r)$, and in particular we can take $C_1=1$.
	\end{enumerate}
\end{lemma}
\begin{proof}
	Part (i) is easy.\ \ Parts (ii) and (iii) are consequences of \cite[Lemma 3.2]{Chiang-Feng1}, which says that for every $0<\alpha\le 1$, there exists $C_\alpha>0$ such that
\begin{align}
	\label{B6} \left|\ln\left|\frac{z_1}{z_2}\right|\right| \le C_\alpha\left(\left|\frac{z_1-z_2}{z_2}\right|^\alpha + \left|\frac{z_2-z_1}{z_1}\right|^\alpha\right)
\end{align}
for every $z_1,z_2\in\mathbb{C}$, and in particular we can take $C_1=1$.
\end{proof}

\pagebreak
\begin{lemma}
\label{Estimate2}
	Let $0<\alpha<1$ be fixed.\ \ Then the following inequalities hold:
	\begin{enumerate}[(i)]
		\item For every $r>0$ and $w\in\mathbb{C}$,
		\[
			\frac{1}{2\pi}\int_0^{2\pi}\frac{d\theta}{|re^{i2\theta}-w|^\alpha} \le \frac{1}{(1-\alpha)r^\alpha}.
		\]
		\item For every $\varepsilon>0$, there exists $M>0$ depending only on $\varepsilon$, such that for every $w\in\mathbb{C}$,
		\[
			\displaystyle\frac{1}{2\pi}\int_0^{2\pi}\frac{d\theta}{|(\sqrt{r}e^{i\theta}+\frac{i}{2})^2-w|^\alpha} \le \frac{1+\varepsilon}{(1-\alpha)r^\alpha}
		\]
		whenever $r>M$.
	\end{enumerate}
\end{lemma}
\begin{proof}
	The proof of (i) is similar to \cite[p. 62]{He-Xiao}, \cite[p. 66]{Jank-Volkmann} or \cite[Lemma 3.3]{Chiang-Feng1}.\ \ To prove (ii), we write $w=\rho e^{i\eta}$, where $\rho\ge0$ and $0\le \eta <2\pi$.\ \ Then
	\begin{align*}
		&\ \ \ \,\frac{1}{2\pi}\int_0^{2\pi}\frac{d\theta}{|(\sqrt{r}e^{i\theta}+\frac{i}{2})^2-w|^\alpha} \\
		&= \frac{1}{2\pi r^\alpha}\int_0^{2\pi}\frac{d\theta}{\left|(e^{i2\theta}+\frac{i}{\sqrt{r}}e^{i\theta}-\frac{1}{4})e^{-i\eta}-\frac{\rho}{r}\right|^\alpha} \\
		&\le \frac{1}{2\pi r^\alpha}\int_0^{2\pi}\frac{d\theta}{|\sin(2\theta-\eta) + \frac{1}{\sqrt{r}}\cos(\theta-\eta) + \frac{1}{4r}\sin\eta|^\alpha} \\
		&= \frac{1}{2\pi r^\alpha}\int_{-\frac{\pi}{4}}^{\frac{7\pi}{4}}\frac{d\theta}{|\sin{2\theta} + \frac{1}{\sqrt{r}}\cos(\theta-\frac{\eta}{2}) + \frac{1}{4r}\sin\eta|^\alpha},
	\end{align*}
	where in the last step we have changed the interval of integration from $[0,2\pi]$ to $I:=[-\frac{\pi}{4},\frac{7\pi}{4}]$ so as to cope with the locations of zeros of the denominator of the integrand.\ \ We note that when $r$ is sufficiently large, the denominator of the last integrand has exactly four zeros in $I$, which we call $\theta_j(r)\in I_j$ for $j\in\{1,2,3,4\}$, where
	\[
		\textstyle I_1:=(-\frac{\pi}{4},\frac{\pi}{4}), \hspace{15pt} I_2:=(\frac{\pi}{4},\frac{3\pi}{4}), \hspace{15pt} I_3:=(\frac{3\pi}{4},\frac{5\pi}{4}), \hspace{15pt} I_4:=(\frac{5\pi}{4},\frac{7\pi}{4}).
	\]

	Now given any $\varepsilon>0$, we claim that there exists $M>0$ depending only on $\varepsilon$, such that whenever $r>M$, we have
	\begin{align}
		\label{B10} \left|\sin{2\theta} + \frac{1}{\sqrt{r}}\cos\left(\theta-\frac{\eta}{2}\right) + \frac{1}{4r}\sin\eta\right| \ge \frac{|\sin{2[\theta-\theta_j(r)]}|}{(1+\varepsilon)^{\frac{1}{\alpha}}}
	\end{align}
	for every $j\in\{1,2,3,4\}$ and $\theta\in I$, with each of the four equalities holding only at the point $\theta=\theta_j(r)$.\ \ Since both $\sin{2\theta} + \frac{1}{\sqrt{r}}\cos\theta + \frac{1}{4r}\sin\eta$ and $\sin{2[\theta-\theta_j(r)]}$ converge pointwise to the function $\sin{2\theta}$ as $r\to\infty$, to show \eqref{B10} it suffices to consider their behaviour around the only (moving) zero $\theta=\theta_j(r)$.

	Since $\theta_j(r)$ is the zero of $\sin{2\theta} + \frac{1}{\sqrt{r}}\cos(\theta-\frac{\eta}{2}) + \frac{1}{4r}\sin\eta$ such that $\displaystyle\lim_{r\to\infty}{\theta_j(r)}=0$, we can easily obtain $\theta_j(r) = O\big(\frac{1}{\sqrt{r}}\big)$ as $r\to\infty$.\ \ Thus around $\theta=\theta_j(r)$, we have the Taylor expansions
	\[
		\sin{2\theta} + \frac{1}{\sqrt{r}}\cos\left(\theta-\frac{\eta}{2}\right) + \frac{1}{4r}\sin\eta = \left(2+\frac{\sin{\frac{\eta}{2}}}{\sqrt{r}}+\cdots\right)[\theta-\theta_j(r)] + \cdots
	\]
	and
	\[
		\sin{2[\theta-\theta_j(r)]} = 2[\theta-\theta_j(r)] + \cdots,
	\]
	from which we see that \eqref{B10} holds for every $\theta\in I$ whenever $r$ is larger than a positive number $M_j$ which depends only on $\varepsilon$.\ \ The proof of the claim is finished by taking $M:=\max{\{M_1,M_2,M_3,M_4\}}$.

	Now \eqref{B10} implies that whenever $r>M$, we have
	\begin{align*}
		\int_{-\frac{\pi}{4}}^{\frac{7\pi}{4}}\frac{d\theta}{|\sin{2\theta} + \frac{1}{\sqrt{r}}\cos(\theta-\frac{\eta}{2}) + \frac{1}{4r}\sin\eta|^\alpha} &\le \sum_{j=1}^4{\int_{I_j}{\frac{1+\varepsilon}{|\sin{2[\theta-\theta_j(r)]}|^\alpha}\,d\theta}} \\
		&= 8(1+\varepsilon)\int_0^{\frac{\pi}{4}}{\frac{d\theta}{\sin^\alpha{2\theta}}},
	\end{align*}
	and so
	\begin{align*}
		\frac{1}{2\pi}\int_0^{2\pi}\frac{d\theta}{|(\sqrt{r}e^{i\theta}+\frac{i}{2})^2-w|^\alpha} &\le \frac{1}{2\pi r^\alpha}\left[8(1+\varepsilon)\int_0^{\frac{\pi}{4}}{\frac{d\theta}{\sin^\alpha{2\theta}}}\right] \\
		&\le \frac{4(1+\varepsilon)}{\pi r^\alpha}\int_0^{\frac{\pi}{4}}\frac{d\theta}{(\frac{4\theta}{\pi})^\alpha} \\
		&= \frac{1+\varepsilon}{(1-\alpha)r^\alpha}.
	\end{align*}
\end{proof}

\noindent \textit{Proof of Theorem~\ref{WilsonLogLemma}.}\ \ Given any meromorphic function $f$, $\D_\W f$ is meromorphic by Proposition~\ref{Mero} (ii), so it makes sense to look for $m\left(r,\frac{\D_\W f}{f}\right)$.\ \ Now the Poisson-Jensen formula asserts that for every $x\in D(0;R)$ which is neither a zero nor a pole of $f$, we have
\begin{align}
	\label{B2}
	\begin{aligned}
		\ln\left|f(x)\right| = &\;\frac{1}{2\pi}\int_{0}^{2\pi}{\ln\left|f(Re^{i\phi})\right|\mathfrak{R}\left(\frac{Re^{i\phi}+x}{Re^{i\phi}-x}\right) d\phi} \\
		&- \sum_{|a_\nu|<R}{\ln\left|\frac{R^2-\overline{a_\nu}x}{R(x-a_\nu)}\right|} + \sum_{|b_\mu|<R}{\ln\left|\frac{R^2-\overline{b_\mu}x}{R(x-b_\mu)}\right|},
	\end{aligned}
\end{align}
where $\{a_\nu\}_\nu$ and $\{b_\mu\}_\mu$ are respectively the sequences of zeros and poles of $f$, repeated according to their multiplicities.\ \ Substituting $x^+=(z+\frac{i}{2})^2$ and $x=z^2$ into \eqref{B2} and subtracting, we obtain
\begin{align}
	\label{B5}
	\begin{aligned}
		\ln\left|\frac{f(x^+)}{f(x)}\right| = &\; \frac{1}{2\pi}\int_{0}^{2\pi}{\ln\left|f(Re^{i\phi})\right|\mathfrak{R}\left[\frac{Re^{i\phi}(2iz-\frac{1}{2})}{(Re^{i\phi}-z^2)(Re^{i\phi}-(z+\frac{i}{2})^2)}\right] d\phi} \\
		&- \sum_{|a_\nu|<R}{\ln\left|\frac{R^2-\overline{a_\nu}(z+\frac{i}{2})^2}{R^2-\overline{a_\nu}z^2}\right|} + \sum_{|b_\mu|<R}{\ln\left|\frac{R^2-\overline{b_\mu}(z+\frac{i}{2})^2}{R^2-\overline{b_\mu}z^2}\right|} \\
		&+ \sum_{|a_\nu|<R}{\ln\left|\frac{(z+\frac{i}{2})^2-a_\nu}{z^2-a_\nu}\right|} - \sum_{|b_\mu|<R}{\ln\left|\frac{(z+\frac{i}{2})^2-b_\mu}{z^2-b_\mu}\right|},
	\end{aligned}
\end{align}
which holds if neither $x$ nor $x^+$ is a zero or a pole of $f$, and if both $x$ and $x^+$ are in $D(0;R)$, i.e. if $R>\frac{1}{4}$ and $0\le r<(\sqrt{R}-\frac{1}{2})^2$, where $r=|x|$. \\

Next we estimate each term on the right-hand side of \eqref{B5} using Lemma~\ref{Estimate1}.\ \ For any given $0<\alpha<1$, the inequality
\begin{align}
	\label{B7}
	\begin{aligned}
		\left|\ln\left|\frac{f(x^+)}{f(x)}\right|\right| \le &\;\frac{R(2\sqrt{r}+\frac{1}{2})}{(R-r)(R-(\sqrt{r}+\frac{1}{2})^2)}\left[m(R,f)+m\left(R,\frac{1}{f}\right)\right] \\
		&+ \frac{2\sqrt{r}+\frac{1}{2}}{R-(\sqrt{r}+\frac{1}{2})^2}\left[n(R,f)+n\left(R,\frac{1}{f}\right)\right] \\
		&+ C_\alpha\left(\sqrt{r}+\frac{1}{4}\right)^\alpha\sum_{|a_\nu|<R}\left[\frac{1}{|z^2-a_\nu|^\alpha} + \frac{1}{|(z+\frac{i}{2})^2-a_\nu|^\alpha}\right] \\
		&+ C_\alpha\left(\sqrt{r}+\frac{1}{4}\right)^\alpha\sum_{|b_\mu|<R}\left[\frac{1}{|z^2-b_\mu|^\alpha} + \frac{1}{|(z+\frac{i}{2})^2-b_\mu|^\alpha}\right]
	\end{aligned}
\end{align}
holds for some constant $C_\alpha>0$.\ \ Noting that
\[
	\left|\ln|w|\right|=\ln^+|w|+\ln^+\left|\frac{1}{w}\right|
\]
for every non-zero complex number $w$, and integrating both sides of \eqref{B7} along the circle $\{z\in\mathbb{C}:|z|=\sqrt{r}\}$, we get
\begin{align}
	\label{B8}
	{\small\begin{aligned}
		&\ 2m\left(r,\frac{f(x^+)}{f(x)}\right)+2m\left(r,\frac{f(x)}{f(x^+)}\right) \\
		\le &\;\frac{R(2\sqrt{r}+\frac{1}{2})}{(R-r)(R-(\sqrt{r}+\frac{1}{2})^2)}\left[m(R,f)+m\left(R,\frac{1}{f}\right)\right] \\
		&+ \frac{2\sqrt{r}+\frac{1}{2}}{R-(\sqrt{r}+\frac{1}{2})^2}\left[n(R,f)+n\left(R,\frac{1}{f}\right)\right] \\
		&+ C_\alpha\left(\sqrt{r}+\frac{1}{4}\right)^\alpha\sum_{|a_\nu|<R}\left[\frac{1}{2\pi}\int_0^{2\pi}\frac{d\theta}{|re^{i2\theta}-a_\nu|^\alpha} + \frac{1}{2\pi}\int_0^{2\pi}\frac{d\theta}{|(\sqrt{r}e^{i\theta}+\frac{i}{2})^2-a_\nu|^\alpha}\right] \\
		&+ C_\alpha\left(\sqrt{r}+\frac{1}{4}\right)^\alpha\sum_{|b_\mu|<R}\left[\frac{1}{2\pi}\int_0^{2\pi}\frac{d\theta}{|re^{i2\theta}-b_\mu|^\alpha} + \frac{1}{2\pi}\int_0^{2\pi}\frac{d\theta}{|(\sqrt{r}e^{i\theta}+\frac{i}{2})^2-b_\mu|^\alpha}\right].
	\end{aligned}}
\end{align}
Note that on the left-hand side we have equivalently integrated along the circle $\{x\in\mathbb{C}:|x|=r\}$ for two complete revolutions.\ \ We emphasize here that the argument $r$ in the proximity functions in \eqref{B8} refers to the modulus of $x$, or of $z^2$. \\

We next give upper bounds to the integrals on the right-hand side of \eqref{B8}.\ \ Applying Lemma~\ref{Estimate2} (i) and Lemma~\ref{Estimate2} (ii) with $\varepsilon=1$, we see that there exists $M>0$ such that
\begin{align}
	\label{B11}
	\begin{aligned}
		&\ 2m\left(r,\frac{f(x^+)}{f(x)}\right) + 2m\left(r,\frac{f(x)}{f(x^+)}\right) \\
		\le &\;\frac{R(2\sqrt{r}+\frac{1}{2})}{(R-r)(R-(\sqrt{r}+\frac{1}{2})^2)}\left[m(R,f)+m\left(R,\frac{1}{f}\right)\right] \\
		&+ \frac{2\sqrt{r}+\frac{1}{2}}{R-(\sqrt{r}+\frac{1}{2})^2}\left[n(R,f)+n\left(R,\frac{1}{f}\right)\right] \\
		&+ C_\alpha\left(\sqrt{r}+\frac{1}{4}\right)^\alpha\left[\sum_{|a_\nu|<R}\frac{1+2}{(1-\alpha)r^\alpha} + \sum_{|b_\mu|<R}\frac{1+2}{(1-\alpha)r^\alpha}\right] \\
		= &\;\frac{R(2\sqrt{r}+\frac{1}{2})}{(R-r)(R-(\sqrt{r}+\frac{1}{2})^2)}\left[m(R,f)+m\left(R,\frac{1}{f}\right)\right] \\
		&+ \left[\frac{2\sqrt{r}+\frac{1}{2}}{R-(\sqrt{r}+\frac{1}{2})^2} + \frac{3C_\alpha\left(\sqrt{r}+\frac{1}{4}\right)^\alpha}{(1-\alpha)r^\alpha}\right]\left[n(R,f)+n\left(R,\frac{1}{f}\right)\right]
	\end{aligned}
\end{align}
whenever $r>M$ and $r<(\sqrt{R}-\frac{1}{2})^2$.\ \ Note that there is no exceptional set for $r$ in \eqref{B11}, as Lemma~\ref{Estimate2} holds for every $w\in\mathbb{C}$ without exception.\ \ Now the number $R>(\sqrt{r}+\frac{1}{2})^2$ was arbitrary in the beginning, so taking $R=2r$ in \eqref{B11} we get
\begin{align}
	\label{B12}
	\begin{aligned}
		&\ m\left(r,\frac{f(x^+)}{f(x)}\right) + m\left(r,\frac{f(x)}{f(x^+)}\right) \\
		\le &\;O\left(\frac{1}{\sqrt{r}}\right)\left[m(2r,f)+m\left(2r,\frac{1}{f}\right)\right] \\
		&+ \left[O\left(\frac{1}{\sqrt{r}}\right) + O\left(\frac{1}{r^\frac{\alpha}{2}}\right)\right]\left[n(2r,f)+n\left(2r,\frac{1}{f}\right)\right] \\
		\le &\;O\left(\frac{1}{\sqrt{r}}\right)T(2r,f) + O\left(\frac{1}{r^\frac{\alpha}{2}}\right)T(4r,f) \\
		\le &\;O\left(\frac{1}{r^\frac{\alpha}{2}}\right)T(4r,f)
	\end{aligned}
\end{align}
as $r\to\infty$, where the second inequality in \eqref{B12} follows from that
\begin{align}
	\label{B13} m(R,f)+m\left(R,\frac{1}{f}\right) \le 2T(R,f)+\ln^+\frac{1}{|f(0)|}
\end{align}
and that
\[
	n(R,f)+n\left(R,\frac{1}{f}\right) \le \frac{2R}{2R-R}\,N(2R,f) \le 2T(2R,f).
\]
Following essentially the same argument, it can also be shown that
\begin{align}
	\label{B14} m\left(r,\frac{f(x^-)}{f(x)}\right)+m\left(r,\frac{f(x)}{f(x^-)}\right) = O\left(\frac{1}{r^\frac{\alpha}{2}}\right)T(4r,f)
\end{align}
as $r\to\infty$. \\

Now if $f\not\equiv0$ is a meromorphic function of finite order $\sigma$, then for every small $0<\varepsilon<1$, we have
\begin{align}
	\label{B15} T(4r,f)=O(r^{\sigma+\frac{\varepsilon}{2}})
\end{align}
as $r\to\infty$.\ \ Choosing $\alpha=1-\varepsilon$, \eqref{B12}, \eqref{B14} and \eqref{B15} imply that
\begin{align}
	\label{B16}
	\begin{aligned}
		&\ \ \ \,m\left(r,\frac{\D_\W f}{f}\right) \\
		&\le m\left(r,\frac{f(x^+)}{f(x)}\right) + m\left(r,\frac{f(x^-)}{f(x)}\right) + 2m\left(r,\frac{1}{iz}\right) \\
		&\le m\left(r,\frac{f(x^+)}{f(x)}\right) + m\left(r,\frac{f(x)}{f(x^+)}\right) + m\left(r,\frac{f(x^-)}{f(x)}\right) + m\left(r,\frac{f(x)}{f(x^-)}\right) + \ln^+\frac{1}{r} \\
		&= O\left(\frac{1}{r^\frac{1-\varepsilon}{2}}\right)T(4r,f) = O(r^{\sigma-\frac{1}{2}+\varepsilon})
	\end{aligned}
\end{align}
as $r\to\infty$, where in the first step we have used the inequalities
\[
	\ln^+{pq} \le \ln^+p + \ln^+q \hspace{25px} \mbox{and} \hspace{25px} \ln^+{\frac{p+q}{2}} \le \ln^+p + \ln^+q,
\]
which hold for every pair of positive real numbers $p$ and $q$.\ \ Of course \eqref{B16} holds for $\varepsilon\ge 1$ as well.\ \ This finishes the proof of Theorem~\ref{WilsonLogLemma}. \hfill \qed \\

\noindent \textbf{Note added in proof.}\ \ After this paper has been completed, we learnt of Korhonen's paper \cite{Korhonen} suggested by the anonymous referee, in which there is an analogue of the lemma on logarithmic derivative for meromorphic functions composed with polynomials.\ \ While the setting is more general in that analogue, our Theorem~\ref{WilsonLogLemma} yields a sharp estimate of $O(r^{\sigma-\frac12+\varepsilon})$ without exceptional set, compared with $O(r^{\sigma-\frac14+\varepsilon})$ with some exceptional set which would have been given by that analogue.\ \ The difference between them mainly comes from our more detailed analysis gearing toward the Wilson operator in Lemma~\ref{Estimate2} (ii).\ \ \cite[Lemma 2.3]{Korhonen} was used to obtain the mentioned analogue, but since it only accounts for the case $0<\alpha<\frac12$ of our Lemma~\ref{Estimate2} (ii), it does not yield our Theorem~\ref{WilsonLogLemma}.\ \ The sharpness of the estimate $O(r^{\sigma-\frac12+\varepsilon})$ in our Theorem~\ref{WilsonLogLemma} is illustrated by the simple example following it, and the estimate also agrees with that for $N(r,f(x^+))$ in Lemma~\ref{Estimate3} in the next section.

\section{Proofs of Main Results}
\label{sec:Proofs}

\subsection{Proof of Theorem~\ref{FundIneq}}

We first prove the fundamental inequality (Theorem~\ref{FundIneq}) by making use of the lemma on logarithmic Wilson difference (Theorem~\ref{WilsonLogLemma}).\ \ The proof is quite standard and is based on Halburd and Korhonen's modification \cite{Halburd-Korhonen} of Nevanlinna's original argument. \\

\noindent \textit{Proof of Theorem~\ref{FundIneq}.}\ \ First let $P(f)$ be the polynomial
	\[
		P(f):=\prod_{n=1}^{q}{(f-y_n)}.
	\]
	We have $\deg_{f}{P}=q>0$.\ \ Since $f\notin\ker\D_\W$, we have $P(f)\not\equiv0$, and so there exist complex numbers $a_1, a_2, \ldots, a_q$ such that
	\[
		\frac{\D_\W f}{P(f)} = \sum_{n=1}^{q}{a_n\frac{\D_\W f}{f-y_n}}.
	\]
	Thus for every $\varepsilon>0$, we have
	\begin{align*}
		m\left(r,\frac{\D_\W f}{P(f)}\right) \le \sum_{n=1}^{q}{m\left(r,\frac{\D_\W f}{f-y_n}\right)} + O(1) = O(r^{\sigma-\frac{1}{2}+\varepsilon}) + O(1)
	\end{align*}
	as $r\to\infty$ by Theorem~\ref{WilsonLogLemma}, which implies that
	\begin{align}
		\label{C3} m\left(r,\frac{1}{P(f)}\right) = m\left(r,\frac{\D_\W f}{P(f)}\frac{1}{\D_\W f}\right) \le m\left(r,\frac{1}{\D_\W f}\right) + O(r^{\sigma-\frac{1}{2}+\varepsilon}) + O(1)
	\end{align}
	as $r\to\infty$, since $f\notin\ker\D_\W$.\ \ Besides this inequality \eqref{C3}, we also have the trivial equality
	\begin{align}
		\label{C4} N\left(r,\frac{1}{P(f)}\right) = \sum_{n=1}^{q}{N\left(r,\frac{1}{f-y_n}\right)}
	\end{align}
	and Mohon'ko's equality \cite{Mohon'ko} \cite[Theorem 2.2.5]{Laine} \cite[Lemma 3.5]{Chiang-Feng1}
	\begin{align}
		\label{C5} T(r,P(f))=qT(r,f)+O(1)
	\end{align}
	as $r\to\infty$. \\

	Since $\frac{1}{P(f)}$ is clearly meromorphic and $\D_\W f$ is meromorphic by Proposition~\ref{Mero}, we can apply Nevanlinna's First Fundamental Theorem \eqref{B1} and the above three results \eqref{C3}, \eqref{C4} and \eqref{C5} to get, for every $\varepsilon>0$,
	\begin{align*}
		T(r,\D_\W f) &= m\left(r,\frac{1}{\D_\W f}\right) + N\left(r,\frac{1}{\D_\W f}\right) + O(1) \\
		&\ge m\left(r,\frac{1}{P(f)}\right) + N\left(r,\frac{1}{\D_\W f}\right) + O(r^{\sigma-\frac{1}{2}+\varepsilon}) + O(1) \\
		&= T(r,P(f)) - N\left(r,\frac{1}{P(f)}\right) + N\left(r,\frac{1}{\D_\W f}\right) + O(r^{\sigma-\frac{1}{2}+\varepsilon}) + O(1) \\
		&= qT(r,f) - \sum_{n=1}^{q}{N\left(r,\frac{1}{f-y_n}\right)} + N\left(r,\frac{1}{\D_\W f}\right) + O(r^{\sigma-\frac{1}{2}+\varepsilon}) + O(1) \\
		&= \sum_{n=1}^{q}{m\left(r,\frac{1}{f-y_n}\right)} + N\left(r,\frac{1}{\D_\W f}\right) + O(r^{\sigma-\frac{1}{2}+\varepsilon}) + O(1)
	\end{align*}
	as $r\to\infty$.\ \ This implies that
	\begin{align*}
		&\ \ \ \ m(r,f)+\sum_{n=1}^{q}{m\left(r,\frac{1}{f-y_n}\right)} \\
		&\le T(r,\D_\W f) - N\left(r,\frac{1}{\D_\W f}\right) + m(r,f) + O(r^{\sigma-\frac{1}{2}+\varepsilon}) + O(1) \\
		&= m(r,\D_\W f) + N(r,\D_\W f) - N\left(r,\frac{1}{\D_\W f}\right) + m(r,f) + O(r^{\sigma-\frac{1}{2}+\varepsilon}) + O(1) \\
		&\le 2m(r,f) + N(r,\D_\W f) - N\left(r,\frac{1}{\D_\W f}\right) + O(r^{\sigma-\frac{1}{2}+\varepsilon}) + O(1) \\
		&\le 2T(r,f) - N_\W(r) + O(r^{\sigma-\frac{1}{2}+\varepsilon}) + O(1)
	\end{align*}
	as $r\to\infty$, where the second last inequality follows from Theorem~\ref{WilsonLogLemma} again. \hfill \qed

\subsection{Wilson counting functions and two lemmas}
Now we introduce the following new counting functions which are more suitable for $\D_\W$.
\begin{definition}
\label{WilsonCount}
	Let $f$ be a meromorphic function such that $f\notin\ker\D_\W$.\ \ We define the following \textbf{\textit{Wilson counting functions}} for $f$:\ \ For each $r\ge 0$, we define
	\begin{enumerate}[(i)]
		\item $n_\W(r,f) := \displaystyle\sum_{x\in S_r}{\mbox{order of zero of }\textstyle\D_\W(\frac{1}{f})\mbox{ at }x^+}$, and
		\item {\small $\widetilde{n_\W}(r,f) := \displaystyle\sum_{x\in S_r}{\max\left\{0,\mbox{order of pole of }f\mbox{ at }x - \mbox{order of zero of }\textstyle\D_\W(\frac{1}{f})\mbox{ at }x^+\right\}}$,}
	\end{enumerate}
	where $S_r:=\left\{x\in \overline{D(0;r)}: \frac{1}{f}\left(x\right)=0\right\}$; and for each $r>0$, we define
	\begin{enumerate}[(i)]
	\setcounter{enumi}{2}
		\item $N_\W(r,f) := \displaystyle\int_{0}^{r}{\frac{n_\W(t,f)-n_\W(0,f)}{t}\,dt} + n_\W(0,f) \ln r$, and
		\item $\widetilde{N_\W}(r,f) := \displaystyle\int_{0}^{r}{\frac{\widetilde{n_\W}(t,f)-\widetilde{n_\W}(0,f)}{t}\,dt} + \widetilde{n_\W}(0,f) \ln r$.
	\end{enumerate}
\end{definition}
Here $\widetilde{n_\W}(r,f)$ and $\widetilde{N_\W}(r,f)$ are respectively the Wilson analogues of the distinct pole counting functions $\overline{n}(r,f)$ and $\overline{N}(r,f)$.\ \ Note that we have defined $\widetilde{n_\W}(r,f)$ in a different way from Halburd and Korhonen's difference operator analogue in \cite{Halburd-Korhonen}, so that this quantity is always non-negative as one should expect, even when there are poles $x$ and $x^{++}$ of $f$ with the same order and the same initial Laurent coefficients.\ \ This rationale has also been adopted in defining the Askey-Wilson analogue of the same quantity in \cite{Chiang-Feng3}.\ \ Also note that if we have used $x^-$ instead of $x^+$ in defining both $n_\W(r,f)$ and of $\widetilde{n_\W}(r,f)$, then it just amounts to the choice of the other branch of the square-root function with the same branch cut in the beginning. \\

Next we state the following useful lemma.
\begin{lemma}
\label{Estimate3}
	Let $f$ be a meromorphic function of finite order $\sigma$. Then for every $\varepsilon>0$, we have
	\[
		N\big(r,f(x^+)\big) = N(r,f) + O(r^{\sigma-\frac{1}{2}+\varepsilon}) + O(\ln r)
	\]
	as $r\to\infty$.
\end{lemma}
\begin{proof}
	It is easy to see that for every $x\in\mathbb{C}$ with $|x|>\frac{1}{4}$, we have $x^{+-}=x^{-+}=x$.\ \ So by the definition of the integrated counting function, we have
	\begin{align*}
		N(r,f(x^+)) &= \int_0^r{\frac{n(t,f(x^+))-n(0,f(x^+))}{t}\,dt} + n(0,f(x^+))\ln r \\
		&= \int_1^r{\frac{n(t,f(x^+))-n(1,f(x^+))}{t}\,dt} + O(\ln r) \\
		&= \sum_{1\le|b_\mu^-|<r}{\ln\frac{r}{|b_\mu^-|}} + O(\ln r)
	\end{align*}
	as $r\to\infty$, where $\{b_\mu\}_\mu$ is the sequence of poles of $f$.\ \ This gives, as $r\to\infty$,
	\begin{align}
		\label{C6}
		\begin{aligned}
			&\ \ \ \,|N(r,f(x^+))-N(r,f)| \\
			&= \left|\sum_{1\le|b_\mu|<r}{\ln\frac{r}{|b_\mu|}}-\sum_{1\le|b_\mu^-|<r}{\ln\frac{r}{|b_\mu^-|}}+O(\ln r)\right| \\
			&\le \sum_{\substack{1\le|b_\mu|<r\\ 1\le|b_\mu^-|<r}}{\left|\ln\left|\frac{b_\mu}{b_\mu^-}\right|\right|} + \sum_{\substack{|b_\mu|\ge r\\ 1\le|b_\mu^-|<r}}{\ln\frac{r}{|b_\mu^-|}} + \sum_{\substack{1\le|b_\mu|<r\\|b_\mu^-|\ge r}}{\ln\frac{r}{|b_\mu|}} + O(\ln r).
		\end{aligned}
	\end{align}

	Next, by applying \eqref{B6} with $\alpha=1$, we see that if $1\le|b_\mu|<r$ and $1\le|b_\mu^-|<r$, then
	\begin{align}
		\label{C7}
		\begin{aligned}
			\left|\ln\left|\frac{b_\mu}{b_\mu^-}\right|\right| &\le \left|\frac{b_\mu-b_\mu^-}{b_\mu}\right| + \left|\frac{b_\mu-b_\mu^-}{b_\mu^-}\right| = \left|\frac{i}{\sqrt{b_\mu}}+\frac{1}{4b_\mu}\right| + \left|\frac{i}{\sqrt{b_\mu^-}}+\frac{1}{4b_\mu^-}\right| \\
			&\le \frac{1}{\sqrt{|b_\mu|}}+\frac{1}{4|b_\mu|} + \frac{1}{\sqrt{|b_\mu^-|}}+\frac{1}{4|b_\mu^-|} \le \frac{5}{4}\left(\frac{1}{\sqrt{|b_\mu|}}+\frac{1}{\sqrt{|b_\mu^-|}}\right).
		\end{aligned}
	\end{align}
	Moreover, if $|b_\mu|\ge r$ and $1\le|b_\mu^-|<r$, then we have
	\begin{align}
		\label{C8}
		\begin{aligned}
			\ln\frac{r}{|b_\mu^-|} &\le \ln\frac{|b_\mu|}{|b_\mu^-|} = \ln\left|1+\frac{i}{2\sqrt{b_\mu^-}}\right|^2 \\
			&\le \ln\left(1+\frac{1}{2\sqrt{|b_\mu^-|}}\right)^2 < \frac{1}{\sqrt{|b_\mu^-|}}.
		\end{aligned}
	\end{align}
	Similarly, if $1\le|b_\mu|<r$ and $|b_\mu^-|\ge r$, then we have
	\begin{align}
		\label{C9} \ln\frac{r}{|b_\mu|} < \frac{1}{\sqrt{|b_\mu|}}.
	\end{align} \\

	Thus by putting \eqref{C7}, \eqref{C8} and \eqref{C9} into \eqref{C6}, we obtain
	\begin{align}
		\label{C10}
		\begin{aligned}
			&\ \ \ \ |N(r,f(x^+))-N(r,f)| \\
			&\le \frac{5}{4}\sum_{\substack{1\le|b_\mu|<r\\ 1\le|b_\mu^-|<r}}{\left(\frac{1}{\sqrt{|b_\mu|}}+\frac{1}{\sqrt{|b_\mu^-|}}\right)} + \sum_{\substack{|b_\mu|\ge r\\ 1\le|b_\mu^-|<r}}{\frac{1}{\sqrt{|b_\mu^-|}}} + \sum_{\substack{1\le|b_\mu|<r\\|b_\mu^-|\ge r}}{\frac{1}{\sqrt{|b_\mu|}}} + O(\ln r) \\
			&\le \frac{5}{4}\left(\sum_{1\le|b_\mu|<r}{\frac{1}{\sqrt{|b_\mu|}}} + \sum_{1\le|b_\mu^-|<r}{\frac{1}{\sqrt{|b_\mu^-|}}}\right) + O(\ln r)
		\end{aligned}
	\end{align}
	as $r\to\infty$.\ \ Now in the second summand on the right-hand side of \eqref{C10}, we have
	\begin{align*}
		\frac{1}{\sqrt{|b_\mu^-|}} = \frac{1}{\sqrt{|b_\mu|}}\sqrt{\left|\frac{b_\mu}{b_\mu^-}\right|} = \frac{1}{\sqrt{|b_\mu|}}\left|1+\frac{i}{2\sqrt{b_\mu^-}}\right| \le \frac{3}{2\sqrt{|b_\mu|}}
	\end{align*}
	since $1\le|b_\mu^-|<r$, thus
	\begin{align}
		\label{C11} |N(r,f(x^+))-N(r,f)| \le \frac{25}{8}\sum_{1\le|b_\mu|<r+\sqrt{r}+\frac{1}{4}}{\frac{1}{\sqrt{|b_\mu|}}} + O(\ln r)
	\end{align}
	as $r\to\infty$. \\

	Finally, to estimate the series on the right-hand side of \eqref{C11}, we consider the following two cases separately:
	\begin{enumerate}[(i)]
		\item If $\sigma<\frac{1}{2}$, then the exponent of convergence of the sequence $\{b_\mu\}_\mu$ is also less than $\frac{1}{2}$, so
		\[
			\sum_{1\le|b_\mu|<r+\sqrt{r}+\frac{1}{4}}{\frac{1}{\sqrt{|b_\mu|}}} = O(1)
		\]
		as $r\to\infty$.
		\item In case $\sigma\ge\frac{1}{2}$, H\"{o}lder's inequality implies that for every $\varepsilon>0$,
		\begin{align*}
			&\ \ \ \ \sum_{1\le|b_\mu|<r+\sqrt{r}+\frac{1}{4}}{\frac{1}{\sqrt{|b_\mu|}}} \\
			&\le \left[\sum_{1\le|b_\mu|<r+\sqrt{r}+\frac{1}{4}}{\left(\frac{1}{\sqrt{|b_\mu|}}\right)^{2\sigma+2\varepsilon}}\right]^{\frac{1}{2\sigma+2\varepsilon}}\left[\sum_{1\le|b_\mu|<r+\sqrt{r}+\frac{1}{4}}{(1)^{\frac{2\sigma+2\varepsilon}{2\sigma-1+2\varepsilon}}}\right]^{\frac{2\sigma-1+2\varepsilon}{2\sigma+2\varepsilon}} \\
			&= \left(\sum_{1\le|b_\mu|<r+\sqrt{r}+\frac{1}{4}}{\frac{1}{|b_\mu|^{\sigma+\varepsilon}}}\right)^{\frac{1}{2\sigma+2\varepsilon}}\left(\sum_{1\le|b_\mu|<r+\sqrt{r}+\frac{1}{4}}{1}\right)^{\frac{2\sigma-1+2\varepsilon}{2\sigma+2\varepsilon}} \\
			&\le O(1)\cdot \left[n\left(r+\sqrt{r}+\frac{1}{4},f\right)\right]^{\frac{2\sigma-1+2\varepsilon}{2\sigma+2\varepsilon}} \\
			&\le O(1)\cdot O\left(\left(r+\sqrt{r}+\frac{1}{4}\right)^{(\sigma+\varepsilon)\left(\frac{2\sigma-1+2\varepsilon}{2\sigma+2\varepsilon}\right)}\right) \\
			&= O(r^{\sigma-\frac{1}{2}+\varepsilon})
		\end{align*}
as $r\to\infty$.
	\end{enumerate}
	Therefore in both cases, \eqref{C11} gives
	\[
		|N(r,f(x^+))-N(r,f)| = O(r^{\sigma-\frac{1}{2}+\varepsilon}) + O(\ln r)
	\]
	as $r\to\infty$.
\end{proof}

Note that the very first sentence in the proof of Lemma~\ref{Estimate3} depends on the fact that we have chosen the branch cut of the square-root function to be the imaginary axis in defining $x^+$. \\

Using the Wilson counting function $N_\W(r,f)$ we have defined, we can interpret the Wilson ramification term $N_\W(r)$ as in the first part of the following lemma.
\begin{lemma}
\label{Estimate4}
	Let $f$ be a meromorphic function of finite order $\sigma$ such that $f\notin\ker\D_\W$.\ \ Then the following inequalities hold:
	\begin{enumerate}[(i)]
		\item For every $r>0$, we have
			\[
				N(r,f) - N_\W(r,f) \le \widetilde{N_\W}(r,f).
			\]
		\item For every $\varepsilon>0$, we have
			\[
				N_\W(r,f) + \sum_{a\in\mathbb{C}}{N_\W\left(r,\frac{1}{f-a}\right)} \le N_\W(r) + O(r^{\sigma-\frac{1}{2}+\varepsilon}) + O(\ln r)
			\]
			as $r\to\infty$, where $N_\W(r)$ is the Wilson ramification term defined in \eqref{C2}.
	\end{enumerate}
\end{lemma}
\begin{proof}
	The first inequality follows immediately from the definitions of the counting functions given in Definition~\ref{WilsonCount}.\ \ To prove the second inequality, we recall that $n_\W(r,f)$ counts all those zeros of $(\D_\W(\frac{1}{f}))(x^+)$ in $\overline{D(0;r)}$ which satisfy $\frac{1}{f}(x)=0$, while for each $a\in\mathbb{C}$, $n_\W(r,\frac{1}{f-a})$ counts all those zeros of $(\D_\W f)(x^+)$ in $\overline{D(0;r)}$ which satisfy $f(x)=a$.\ \ So we have
	\begin{align*}
		&\ \ \ \ n_\W(r,f) + \sum_{a\in\mathbb{C}}{n_\W\left(r,\frac{1}{f-a}\right)} \\
		&\le n\left(r,\frac{1}{(\D_\W(\frac{1}{f}))(x^+)}\right) + n\left(r,\frac{1}{(\D_\W f)(x^+)}\right) \\
		&\le n(r,f) + n(r,f(x^{++})) + n\left(r,\frac{1}{(\D_\W f)(x^+)}\right) - n(r,(\D_\W f)(x^+))
	\end{align*}
	as
	\[
		\left(\D_\W\left(\frac{1}{f}\right)\right)(x^+) = \frac{-(\D_\W f)(x^+)}{f(x)f(x^{++})} \hspace{20px} \mbox{and} \hspace{20px} (\D_\W f)(x^+) = \frac{f(x^{++})-f(x)}{x^{++}-x}.
	\]
	Therefore for every $\varepsilon>0$,
	\begin{align*}
		&\ \ \ \ N_\W(r,f) + \sum_{a\in\mathbb{C}}{N_\W\left(r,\frac{1}{f-a}\right)} \\
		&\le N(r,f) + N(r,f(x^{++})) + N\left(r,\frac{1}{(\D_\W f)(x^+)}\right) - N(r,(\D_\W f)(x^+)) \\
		&= 2N(r,f) + N\left(r,\frac{1}{\D_\W f}\right) - N(r,\D_\W f) + O(r^{\sigma-\frac{1}{2}+\varepsilon}) + O(\ln r) \\
		&= N_\W(r) + O(r^{\sigma-\frac{1}{2}+\varepsilon}) + O(\ln r)
	\end{align*}
	as $r\to\infty$, where the second step follows from Lemma~\ref{Estimate3}, as the order $\sigma$ of $f$ is greater than or equal to the order of $\frac{1}{\D_\W f}$.
\end{proof}

\pagebreak
\subsection{Proof of Theorem~\ref{NSFTWilson}}
\begin{proof}
	From Lemma~\ref{Estimate4} (i), we see that $N(r,\frac{1}{f-y_n})-\widetilde{N_\W}(r,\frac{1}{f-y_n}) \le N_\W(r,\frac{1}{f-y_n})$ for each $1\le n\le q$ and that $N(r,f)-\widetilde{N_\W}(r,f) \le N_\W(r,f)$, so together with Lemma~\ref{Estimate4} (ii) we get, for every $\varepsilon>0$,
	\begin{align}
		\label{C13}
		\begin{aligned}
			&\ \ \ \ [N(r,f)-\widetilde{N_\W}(r,f)] + \sum_{n=1}^q{\left[N\left(r,\frac{1}{f-y_n}\right)-\widetilde{N_\W}\left(r,\frac{1}{f-y_n}\right)\right]} \\
			&\le N_\W(r,f) + \sum_{n=1}^q{N_\W\left(r,\frac{1}{f-y_n}\right)} \le N_\W(r,f) + \sum_{a\in\mathbb{C}}{N_\W\left(r,\frac{1}{f-a}\right)} \\
			&\le N_\W(r) + O(r^{\sigma-\frac{1}{2}+\varepsilon}) + O(\ln r)
		\end{aligned}
	\end{align}
	as $r\to\infty$.\ \ Now, applying Nevanlinna's First Fundamental Theorem \eqref{B1} and this inequality \eqref{C13} to the result \eqref{C1} from Theorem~\ref{FundIneq}, we have
	\begin{align*}
		(q-1)T(r,f) &\le N(r,f)+\sum_{n=1}^{q}{N\left(r,\frac{1}{f-y_n}\right)} - N_\W(r) + O(r^{\sigma-\frac{1}{2}+\varepsilon}) + O(\ln r) \\
		&\le \widetilde{N_\W}(r,f)+\sum_{n=1}^{q}{\widetilde{N_\W}\left(r,\frac{1}{f-y_n}\right)} + O(r^{\sigma-\frac{1}{2}+\varepsilon}) + O(\ln r)
	\end{align*}
	as $r\to\infty$.
\end{proof}

\subsection{Defect relations under the Wilson operator, proofs of Corollary~\ref{Defect} and Theorem~\ref{WilsonPicard}}

We can apply Theorem~\ref{NSFTWilson} to give a Wilson analogue of Nevanlinna's Theorem on deficient values.\ \ We first recall that given a meromorphic function $f$, the Nevanlinna defect $\delta(a,f)$ of $f$ at $a\in\hat{\mathbb{C}}$ is defined as
\[
	\delta(a,f) := \liminf_{r\to\infty}\frac{m(r,a)}{T(r,f)} = 1-\limsup_{r\to\infty}\frac{N(r,a)}{T(r,f)},
\]
where we have used the short-hand notations
\[
	m(r,a) := m\left(r,\frac{1}{f-a}\right) \hspace{25px} \mbox{and} \hspace{25px} N(r,a) := N\left(r,\frac{1}{f-a}\right)
\]
for every $a\in\mathbb{C}$, and
\[
	m(r,\infty) := m(r,f) \hspace{25px} \mbox{and} \hspace{25px} N(r,\infty) := N(r,f),
\]
and where the equality after the definition follows from Nevanlinna's First Fundamental Theorem \eqref{B1}.\ \ Now we make the following natural definitions, which are the Wilson analogues of the ramification index $\vartheta(a,f)$ of $f$ at $a$ and the quantity $\Theta(a,f)$ in the classical Nevanlinna theory.

\begin{definition}
\label{WTheta}
	Let $f$ be a meromorphic function such that $f\notin\ker\D_\W$ and let $a\in\hat{\mathbb{C}}$.\ \ The \textbf{\textit{Wilson ramification index}} $\vartheta_\W(a,f)$ of $f$ at $a$ is defined as
	\[
		\vartheta_\W(a,f):=\liminf_{r\to\infty}\frac{N_\W(r,a)}{T(r,f)},
	\]
	and the quantity $\Theta_\W(a,f)$ of $f$ at $a$ is defined as
	\[
		\Theta_\W(a,f):= 1-\limsup_{r\to\infty}\frac{\widetilde{N_\W}(r,a)}{T(r,f)}.
	\]
\end{definition}

\bigskip
It is easy to see that we always have $0\le\Theta_\W(a,f)\le 1$.\ \ Now we prove Corollary~\ref{Defect}. \\

\noindent \textit{Proof of Corollary~\ref{Defect}.}\ \ The first inequality follows readily from the first part of Lemma~\ref{Estimate4}.\ \ Now from Theorem~\ref{NSFTWilson}, dividing both sides of \eqref{C12} by $T(r,f)$, we see that for any positive integer $q$, any finite sequence of points $\{y_1,y_2,\ldots,y_q\}$ in $\mathbb{C}$ and any $\varepsilon>0$,
	\[
		q-1 \le \frac{\widetilde{N_\W}(r,f)}{T(r,f)} + \sum_{n=1}^q{\frac{\widetilde{N_\W}(r,y_n)}{T(r,f)}} + O\left(\frac{r^{\sigma-\frac{1}{2}+\varepsilon}}{T(r,f)}\right) + O\left(\frac{\ln r}{T(r,f)}\right)
	\]
	as $r\to\infty$.\ \ Rearranging the terms, we then obtain
	\[
		\left[1 - \frac{\widetilde{N_\W}(r,f)}{T(r,f)}\right] + \sum_{n=1}^q{\left[1 - \frac{\widetilde{N_\W}(r,y_n)}{T(r,f)}\right]} \le 2 + O\left(\frac{r^{\sigma-\frac{1}{2}+\varepsilon}}{T(r,f)}\right) + O\left(\frac{\ln r}{T(r,f)}\right)
	\]
	as $r\to\infty$. Taking limit inferior on both sides as $r\to\infty$ and noting that $f$ is transcendental, we have
	\[
		\Theta_\W(\infty,f) + \sum_{n=1}^q{\Theta_\W(y_n,f)} \le 2.
	\]
	The rest of the corollary follows since $\{y_1,y_2,\ldots,y_q\}$ is an arbitrary sequence of points in $\mathbb{C}$. \hfill \qed \\

As a Wilson analogue of the Nevanlinna exceptional value, we make the following definition.
\begin{definition}
\label{ExVWilson}
	Let $f$ be a meromorphic function and $a$ be an extended complex number.
	\begin{enumerate}[(i)]
		\item We say that $f$ has a \textbf{\textit{Wilson $\boldsymbol{a}$-sequence}} starting at $x_0\in\mathbb{C}$ if for every non-negative integer $p$, $f$ has an $a$-point of multiplicity $m_p>0$ at $x_0^{+(2p)}$, where
		\begin{itemize}
			\item The sequence $\{m_p\}_{p=0}^\infty$ is monotonically increasing, i.e. $0<m_0\le m_1\le m_2\le \cdots$; and
			\item In case $x_0^{--}\ne x_0$ (i.e. $x_0\notin[-\frac{1}{4},0]$), we also have either $f(x_0^{--})\ne a$ or $f(x_0^{--})=a$ with multiplicity $m_{-1}>m_0$.
		\end{itemize}
		This Wilson $a$-sequence is given by $\{x_0^{+(2p)}\}_{p=0}^\infty$.
		\item We say that $a$ is a \textbf{\textit{Wilson exceptional value}} of $f$ if all but at most finitely many $a$-points of $f$ are in its Wilson $a$-sequences.
	\end{enumerate}
\end{definition}
We should note that the Wilson exceptional values of a function is a global property of the function.\ \ Moreover, if $f$ is a rational function, i.e. $f$ has finitely many poles, then $\infty$ is automatically a Wilson exceptional value of $f$.

\pagebreak
\begin{example}
Here is an example that illustrates Definition~\ref{ExVWilson}.\ \ Let $f$ be a meromorphic function having poles at the points as shown in the following figure, with multiplicities indicated in boxes.\ \ The dots indicate that the subsequent poles have increasing multiplicities.
\end{example}

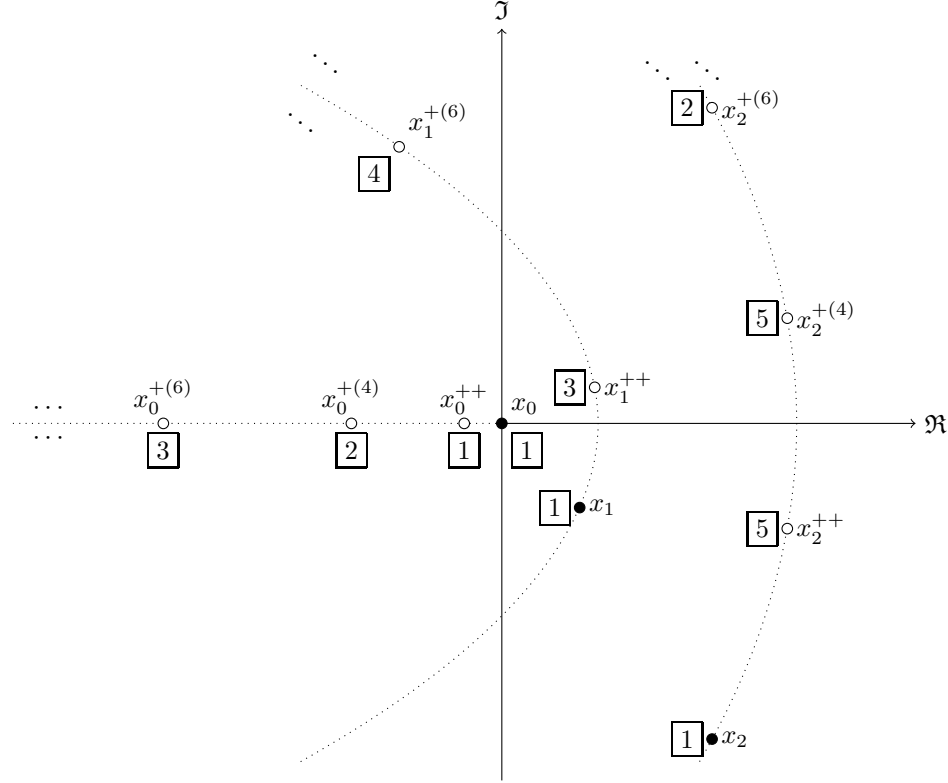
\begin{center}
\begin{tikzpicture}[scale=0.5]
	\draw[dotted] (-13,0) -- (0,0);
	\draw[->] (0,0) -- (11,0) node[right] {$\mathfrak{R}$};
	\draw[->] (0,-9.5) -- (0,10.5) node[above] {$\mathfrak{I}$};

	\draw[fill] (0,0) circle[radius=4pt] node[above right] {$x_0$} node[below right] {\fbox{1}};
	\draw[fill=white] (-1,0) circle[radius=4pt] node[above] {$x_0^{++}$} node[below] {\fbox{1}};
	\draw[fill=white] (-4,0) circle[radius=4pt] node[above] {$x_0^{+(4)}$} node[below] {\fbox{2}};
	\draw[fill=white] (-9,0) circle[radius=4pt] node[above] {$x_0^{+(6)}$} node[below] {\fbox{3}};
	\draw[fill=white] (-12,0) circle[radius=0pt] node[above] {$\cdots$} node[below] {$\cdots$};

	\draw[dotted, rotate=90] (-9,5.35015625) parabola bend (0,-2.56) (9,5.35015625);
	\draw[fill] (2.07,-2.24) circle[radius=4pt] node[right] {$x_1$} node[left] {\fbox{1}};
	\draw[fill=white] (2.47,0.96) circle[radius=4pt] node[right] {$x_1^{++}$} node[left] {\fbox{3}};
	\draw[fill=white] (-2.73,7.36) circle[radius=4pt] node[above right] {$x_1^{+(6)}$} node[below left] {\fbox{4}};
	\draw[fill=white] (-5.35015625,9) circle[radius=0pt] node[above right] {$\ddots$} node[below] {$\ddots$};

	\draw[dotted, rotate=90] (-9,-5.257092) parabola bend (0,-7.84) (9,-5.257092);
	\draw[fill] (5.59,-8.4) circle[radius=4pt] node[right] {$x_2$} node[left] {\fbox{1}};
	\draw[fill=white] (7.59,-2.8) circle[radius=4pt] node[right] {$x_2^{++}$} node[left] {\fbox{5}};
	\draw[fill=white] (7.59,2.8) circle[radius=4pt] node[right] {$x_2^{+(4)}$} node[left] {\fbox{5}};
	\draw[fill=white] (5.59,8.4) circle[radius=4pt] node[right] {$x_2^{+(6)}$} node[left] {\fbox{2}};
	\draw[fill=white] (4.8,9.6) circle[radius=0pt] node[right] {$\ddots$} node[left] {$\ddots$};
\end{tikzpicture}
\end{center}

From this figure, we see that $f$ has Wilson pole sequences starting at the points $x_0$, $x_1^{+(6)}$ and $x_2^{+(6)}$.\ \ (Note that $x_0^{--}=x_0$.)\ \ Since all but five poles of $f$ are contained in these Wilson pole sequences, it follows that $\infty$ is a Wilson exceptional value of $f$. \\

\begin{remark}
\label{ExVWilsonRmk}
	In fact, for any Wilson exceptional value $a$ of a transcendental meromorphic function $f$, we have $\Theta_\W(a,f)=1$.\ \ This is because each $a$-point $x_0$ of $f$ in a Wilson $a$-sequence has the same multiplicity as that of the zero of $\D_\W f$ at $x_0^+$, which means that each Wilson $a$-sequence of $f$ gives no contribution to $\widetilde{n_\W}(r,a)$.\ \ Now $f$ has only finitely many $a$-points outside its Wilson $a$-sequences, so $\widetilde{n_\W}(r,a)$ is bounded.\ \ Therefore $\widetilde{N_\W}(r,a)=O(\ln r)=o(T(r,f))$ as $r\to\infty$, and so $\Theta_\W(a,f) = 1$.
\end{remark}

\pagebreak
The following is the proof of Theorem~\ref{WilsonPicard}. \\

\noindent \textit{Proof of Theorem~\ref{WilsonPicard}.}\ \ Let $f$ be a meromorphic function and $a_1,a_2,a_3\in\hat{\mathbb{C}}$ be three distinct Wilson exceptional values of $f$.\ \ If $f$ is rational, then we are done.\ \ If $f$ is transcendental, then
	\[
		\Theta_\W(a_1,f) = \Theta_\W(a_2,f) = \Theta_\W(a_3,f) = 1
	\]
	as in the above remark.\ \ This implies that
	\[
		\sum_{i=1}^3{\Theta_\W(a_i,f)} = 3 > 2,
	\]
	which contradicts the result of Corollary~\ref{Defect}, so $f\in\ker\D_\W$. \hfill \qed \\

Theorem~\ref{WilsonPicard} readily implies the following corollary, which is a Wilson analogue of Liouville's theorem in classical complex function theory.

\begin{corollary}
	Let $f$ be a meromorphic function of finite order.\ \ If for every $a\in\hat{\mathbb{C}}$ outside a sufficiently large disk, the set
	\[
		\{x\in f^{-1}(\{a\}): \mbox{$a$-point multiplicity of $f$ at $x^{++}$}<\mbox{$a$-point multiplicity of $f$ at $x$} \}
	\]
	is finite, then either $f\in\ker{\D_\W}$ or $f$ is rational.
\end{corollary}

\bigskip
We conclude this section by giving some examples.
\begin{example}
\label{defectegs}
	Here are examples on different values of $\Theta_\W(a,f)$.
	\begin{enumerate}[(i)]
	\item Let $b\in\mathbb{C}\setminus(i\mathbb{N}_0)$, and $f:\mathbb{C}\to\mathbb{C}$ be the infinite product
		\[
			f(x):=\prod_{k=0}^{\infty}{\left[1-\frac{x}{(b+ki)^2}\right]}=\frac{[\Gamma(-ib)]^2}{\Gamma(-ib+i\sqrt{x})\Gamma(-ib-i\sqrt{x})}.
		\]
		Since $f$ has a simple zero at $(b+ki)^2$ for every $k\in\mathbb{N}_0$, $0$ is a Wilson exceptional value of $f$, and so $\Theta_\W(0,f)=1$.\ \ Since $f$ is entire, we also have $\Theta_\W(\infty,f)=1$, so the equality $\sum\Theta_\W=2$ is attained.\ \ Note that by Gauss' formula \cite[\S 1.3]{Bailey} we have
		\[
			f(x) = \,_2F_1\left(
			\begin{matrix}
			\begin{matrix}
				i\sqrt{x}, & -i\sqrt{x}
			\end{matrix} \\
				-ib
			\end{matrix}
			\ ;\ 1
			\right)
		\]
		if $\mathfrak{I}b>0$.\ \ Replacing $i$ by any $c\in\mathbb{C}\setminus\{0\}$, we see that the function $f:\mathbb{C}\to\mathbb{C}$ defined by
		\begin{align*}
			f(x) &:=\prod_{k=0}^{\infty}{\left[1-\frac{x}{(b+ck)^2}\right]}=\frac{[\Gamma(\frac{b}{c})]^2}{\Gamma(\frac{b}{c}+\frac{1}{c}\sqrt{x})\Gamma(\frac{b}{c}-\frac{1}{c}\sqrt{x})} \\
			&= \,_2F_1\left(
			\begin{matrix}
			\begin{matrix}
				\frac{1}{c}\sqrt x, & -\frac{1}{c}\sqrt x
			\end{matrix} \\
				\frac{b}{c}
			\end{matrix}
			\ ;\ 1
			\right)
		\end{align*}
		satisfies that $\Theta_{\W,2}(\infty,f)=\Theta_{\W,2}(0,f)=1$.\ \ Here $\Theta_{\W,c}(a,f)$ is the analogous quantity of $f$ at $a$ associated to the \textit{$c$-shift Wilson operator} $\D_{\W,c}$ to be defined in \S\ref{sec:Ptwise}.\ \ In particular, this is true for the function $f(x)=\cos{\frac{\pi}{2}\sqrt{x}}$, to which we have substituted $b=1$ and $c=2$.
	\item Consider the generating function $\varphi(x;t)$ of Wilson polynomials $W_n(x)\equiv W_n(x;a,b,c,d)$ \cite{Wilson2} \cite{Koekoek-Swarttouw} given by
		\[
			\varphi(x;t) := \,_2F_1\left(
			\begin{matrix}
			\begin{matrix}
				a+i\sqrt{x}, & b+i\sqrt{x}
			\end{matrix} \\
				a+b
			\end{matrix}
			\ ;\ t
			\right)\,_2F_1\left(
			\begin{matrix}
			\begin{matrix}
				c-i\sqrt{x}, & d-i\sqrt{x}
			\end{matrix} \\
				c+d
			\end{matrix}
			\ ;\ t
			\right)
		\]
		When $a=c\notin\{\frac{1}{2}-k:k\in\mathbb{N}\}$ and $b=d=\frac{1}{2}$, we can evaluate $\varphi(x;t)$ at $t=-1$ using Kummer's formula \cite[\S 2.3]{Bailey} to obtain
		\begin{align*}
			\varphi(x) &= \,_2F_1\left(
			\begin{matrix}
			\begin{matrix}
				a+i\sqrt{x}, & \frac{1}{2}+i\sqrt{x}
			\end{matrix} \\
				a+\frac{1}{2}
			\end{matrix}
			\ ;\ -1
			\right)\,_2F_1\left(
			\begin{matrix}
			\begin{matrix}
				a-i\sqrt{x}, & \frac{1}{2}-i\sqrt{x}
			\end{matrix} \\
				a+\frac{1}{2}
			\end{matrix}
			\ ;\ -1
			\right) \\
			&= \frac{\Gamma(a+\frac{1}{2})\Gamma(a+\frac{1}{2})\Gamma(1+\frac{a}{2}+\frac{i}{2}\sqrt{x})\Gamma(1+\frac{a}{2}-\frac{i}{2}\sqrt{x})}{\Gamma(1+a+i\sqrt{x})\Gamma(1+a-i\sqrt{x})\Gamma(\frac{1}{2}+\frac{a}{2}+\frac{i}{2}\sqrt{x})\Gamma(\frac{1}{2}+\frac{a}{2}-\frac{i}{2}\sqrt{x})}.
		\end{align*}
		This function $\varphi$ has double zeros at $\{-(a-1+2k)^2:k\in\mathbb{N}\}$ and has no other zeros or poles.\ \ We have $\Theta_{\W,2i}(0,\varphi)=\Theta_{\W,2i}(\infty,\varphi)=1$.
	\end{enumerate}
	The next two examples make use of the infinite product construction as in the first example.
	\begin{enumerate}[(i)]
	\setcounter{enumi}{2}
	\item Let $g:\mathbb{C}\to\mathbb{C}$ be the infinite product
		\[
			g(x):=\prod_{k=1}^{\infty}{\left[1-\frac{x}{(2ki)^2}\right]^2\left[1-\frac{x}{((2k-1)i)^2}\right]}.
		\]
		Then we have $n(r,0)\sim\frac{3}{2}\sqrt r$ and $\widetilde{n_\W}(r,0)\sim\frac{1}{2}\sqrt r$ as $r\to\infty$, so
		\[
			\widetilde{N_\W}(r,0) \sim \sqrt r
		\]
		as $r\to\infty$.\ \ On the other hand, it is easy to see that for every $\theta\in(-\pi,\pi)$,
		\[
			\ln |g(re^{i\theta})| = \frac{3\pi}{2}\sqrt r\cos\frac{\theta}{2}
		\]
		as $r\to\infty$, so by Lebesgue Dominated Convergence Theorem we have
		\[
			T(r,g) = m(r,g) \sim \frac{1}{2\pi}\frac{3\pi}{2}\sqrt r\int_{-\pi}^\pi{\cos\frac{\theta}{2}\,d\theta} = 3\sqrt r
		\]
		as $n\to\infty$.\ \ Therefore $\Theta_\W(0,g)=\frac{2}{3}$. \\
		More generally, given any rational number $s\in[0,1]$, there exist non-negative integers $p\ge q$ such that $s=\frac{2q}{p+q}$.\ \ If we let $g:\mathbb{C}\to\mathbb{C}$ be the infinite product
		\[
			g(x):=\prod_{k=1}^{\infty}{\left[1-\frac{x}{(2ki)^2}\right]^p\left[1-\frac{x}{((2k-1)i)^2}\right]^q},
		\]
		then we have $\Theta_\W(0,g)=s$.
	\item Given any $s\in[0,1]$, there exists a sequence $\{a_k\}_{k=0}^\infty$ of integers such that $a_0=0$, $a_{k+1}-a_k\in\{0,1\}$ for every $k\in\mathbb{N}_0$, and $\displaystyle\lim_{k\to\infty}\frac{a_k}{k}=1-s$.\ \ If we let $h:\mathbb{C}\to\mathbb{C}$ be the infinite product
		\[
			h(x):=\prod_{k=1}^{\infty}{\left[1-\frac{x}{((k+a_k)i)^2}\right]},
		\]
		then by a similar argument we obtain $\Theta_\W(0,h)=s$.
	\end{enumerate}
\end{example}

\subsection{Unicity theorem}

As in the classical Nevanlinna theory, Theorem~\ref{NSFTWilson} also gives us a Wilson analogue to Nevanlinna's five-value theorem.\ \ To state this theorem we first need to define what it means by two functions ``sharing" a value in the Wilson sense.

\begin{definition}
	Let $f$ and $g$ be meromorphic functions of finite order such that $f,g\notin\ker\D_\W$ and let $a\in\hat{\mathbb{C}}$.\ \ We say that $f$ and $g$ \textbf{\textit{share}} the value $a$ \textbf{\textit{in the Wilson sense}} if, as $r\to\infty$,
	\[
		\widetilde{n_\W}(r,f,a)-\widetilde{n_\W}(r,g,a)=O(1),
	\]
	or equivalently,
	\[
		\widetilde{N_\W}(r,f,a)-\widetilde{N_\W}(r,g,a)=O(\ln r).
	\]
\end{definition}

\begin{theorem}
\label{N5Wilson}
\textup{(Nevanlinna's five-value theorem for the Wilson operator)}
	Let $f$ and $g$ be meromorphic functions of finite order such that $f,g\notin\ker\D_\W$.\ \ If $f$ and $g$ share five distinct values $a_1, a_2, a_3, a_4, a_5\in\hat{\mathbb{C}}$ in the Wilson sense, then either $f\equiv g$ or $f$ and $g$ are both rational.
\end{theorem}
\begin{proof}
	The proof is similar to the classical one in Hayman \cite{Hayman}.\ \ From the assumption, we have
	\[
		N_j(r) := \widetilde{N_\W}(r,f,a_j) = \widetilde{N_\W}(r,g,a_j) + O(\ln r)
	\]
	as $r\to\infty$ for each $j=1, 2, 3, 4, 5$.\ \ If one of $f$ and $g$ is rational, then so is the other because $\Theta_\W(a_j,f)=\Theta_\W(a_j,g)$ for $j=1, 2, 3, 4, 5$.\ \ If $f$ and $g$ are both transcendental, then by Theorem~\ref{NSFTWilson}, no matter $\infty$ is one of the $a_j$'s or not, we have for every $\varepsilon>0$,
	\begin{align*}
		3T(r,f) &\le \sum_{j=1}^{5}{\widetilde{N_\W}\left(r,f,a_j\right)} + O(r^{\sigma_f-\frac{1}{2}+\varepsilon}) + O(\ln r) \\
		3T(r,g) &\le \sum_{j=1}^{5}{\widetilde{N_\W}\left(r,g,a_j\right)} + O(r^{\sigma_g-\frac{1}{2}+\varepsilon}) + O(\ln r)
	\end{align*}
	as $r\to\infty$.\ \ As a result, there exists a sequence of positive real numbers $\{r_n\}_{n=1}^\infty$ strictly increasing to infinity such that
	\[
		(3+o(1))T(r_n,\, f) \le \sum_{j=1}^{5}{N_j(r_n)}\hspace{20px}\mbox{and}\hspace{20px}(3+o(1))T(r_n,\, g) \le \sum_{j=1}^{5}{N_j(r_n)}
	\]
	as $n\to\infty$.\ \ Now if we suppose on the contrary that $f\not\equiv g$, then we obtain
	\begin{align*}
		T\left(r_n,\frac{1}{f-g}\right) &= T(r_n, f-g) + O(1) \\
		&\le T(r_n,f) + T(r_n,g) + O(1) \\
		&\le \left(\frac{2}{3}+o(1)\right)\sum_{j=1}^{5}{N_j(r_n)}
	\end{align*}
	as $n\to\infty$.\ \ On the other hand, we also have
	\[
		\sum_{j=1}^{5}{N_j(r_n)} \le \widetilde{N_\W}\left(r_n,\frac{1}{f-g}\right) \le T\left(r_n,\frac{1}{f-g}\right) + O(1)
	\]
	as $n\to\infty$.\ \ These inequalities imply that
	\[
		\sum_{j=1}^{5}{N_j(r_n)} = O(1)
	\]
	as $n\to\infty$, which is impossible since $f$ and $g$ are transcendental and $f,g\notin\ker\D_\W$.
\end{proof}

\section{Applications in Wilson Difference Equations and Wilson Interpolation Equations}
\label{sec:Applications}

In addition to the Wilson analogue of Nevanlinna's Second Fundamental Theorem and defect relations, we can also derive from the lemma on logarithmic Wilson difference some useful results about Wilson difference equations and Wilson interpolation equations.\ \ An (ordinary) Wilson difference equation is an equation involving an unknown complex function and its Wilson differences, i.e. an equation of the form
\[
	F(x, y, \D_\W y, \D_\W^2 y, \D_\W^3 y, \ldots) = 0.
\]

An (ordinary) Wilson interpolation equation is, on the other hand, an equation involving an unknown complex function and its Wilson shifts, i.e. an equation of the form
\[
	F(x, y(x), y(x^+), y(x^-), y(x^{++}), y(x^{--}), \ldots) = 0.
\]

\subsection{Wilson-Sturm-Liouville operator}
It is well-known \cite{Koekoek-Swarttouw} that the Wilson polynomial $y(z)=W_n(z^2;a,b,c,d)$ of degree $n$ (in the variable $z^2$) satisfies the second order difference equation
\begin{align*}
	n(n+a+b+c+d-1)y(z) = \ &\frac{(a-iz)(b-iz)(c-iz)(d-iz)}{(2iz)(2iz-1)}[y(z+i)-y(z)] \\
	&- \frac{(a+iz)(b+iz)(c+iz)(d+iz)}{(2iz)(2iz+1)}[y(z)-y(z-i)],
\end{align*}
which can be written into the form
\begin{align}
	\label{C13a}
	\begin{aligned}
		&\D_\W\left(\mu\left(x;a+\frac{1}{2},b+\frac{1}{2},c+\frac{1}{2},d+\frac{1}{2}\right)\D_\W W_n\right)(x) \\
		&\qquad + n(n+a+b+c+d-1)\mu(x;a,b,c,d)W_n(x)=0,
	\end{aligned}
\end{align}
where
\[
	\mu(x;a,b,c,d):=\frac{\Gamma(a-iz)\cdots\Gamma(d-iz)\Gamma(a+iz)\cdots\Gamma(d+iz)}{\Gamma(2iz)\Gamma(-2iz)}
\]
is a weight function for the orthogonality of Wilson polynomials $W_n(x)=W_n(z^2)$, provided that the sum of any two (not necessarily distinct) elements in $\{a,b,c,d\}$ is not a non-positive integer \cite{Wilson}.\ \ In fact, Wilson polynomials $W_n$ are at the highest level in the Askey scheme of hypergeometric orthogonal polynomials \cite{Koekoek-Swarttouw}, and many hypergeometric orthogonal polynomials which occur in physics can be obtained from $W_n$ by taking limits of the parameters $a,b,c,d$. \\

Wilson polynomials themselves also occur naturally in physics.\ \ In investigating the Lorentz transformation properties of the parity operator $\mathcal{P}$, Bender, Meisinger and Wang arrived at an equation \cite[(24)]{Bender-Meisinger-Wang}, which can be shown to reduce to a (second-order) Wilson difference equation
\begin{align}
	\label{C14b}
	\begin{aligned}
		&\D_{\W,1}\left(\mu_1\left(x;a+\frac{1}{2},b+\frac{1}{2},c+\frac{1}{2},d+\frac{1}{2}\right)\D_{\W,1}g\right)(x) \\
		&\qquad +(1-l_1^2)\mu_1(x;a,b,c,d)g(x)=0,
	\end{aligned}
\end{align}
where $a,b,c,d$ are complex constants (in fact $a=\frac{1}{2}$, $b\in\mathbb{R}$ and $c=\overline{d}$), $\D_{\W,1}$ is the \textit{$1$-shift Wilson operator} to be defined in \S\ref{sec:Ptwise}, and
\[
	\mu_1(x;a,b,c,d):=\frac{\Gamma(a-z)\cdots\Gamma(d-z)\Gamma(a+z)\cdots\Gamma(d+z)}{\Gamma(2z)\Gamma(-2z)}
\]
is also a weight function for the orthogonality of Wilson polynomials, provided that the sum of any two (not necessarily distinct) elements in $\{a,b,c,d\}$ is not a non-positive integer \cite{Wilson}.\ \ This Wilson difference equation \eqref{C14b} has eigenvalues $l_1(n)=2n+1$, where $n\in\mathbb{N}_0$, and the corresponding eigensolutions are $g_0(x)=1$ and
\[
	g_n(x) = \left(x-\frac{1}{4}\right)\frac{(n+1)!}{(2n)!}\,W_{n-1}\left(-x;\frac{1}{2},\frac{1}{2},\frac{3}{2},\frac{3}{2}\right), \hspace{25px} \mbox{for $n\in\mathbb{N}$.}
\]

\bigskip
The operator
\[
	\L_\W:=\frac{1}{\mu(x;a,b,c,d)}\D_\W\mu\left(x;a+\frac{1}{2},b+\frac{1}{2},c+\frac{1}{2},d+\frac{1}{2}\right)\D_\W
\]
is a Wilson analogue of the Sturm-Liouville operator.\ \ Using this notation, \eqref{C13a} becomes
\[
	\L_\W W_n = -n(n+a+b+c+d-1)W_n,
\]
so $W_n$ is an eigenfunction of $\L_\W$ corresponding to the eigenvalue $\lambda_n=-n(n+a+b+c+d-1)$. \\

To investigate Wilson difference equations of the form $\L_\W y=\phi y$, we first obtain the following using the lemma on logarithmic Wilson difference.
\begin{lemma}
If $f\not\equiv0$ is a meromorphic function of finite order $\sigma$, then for every $\varepsilon>0$,
\[
	m\left(r,\frac{\L_\W f}{f}\right) = O(r^{\max\{0,\sigma-\frac{1}{2}\}+\varepsilon})
\]
as $r\to\infty$.
\end{lemma}
\begin{proof}
	If $f\not\equiv0$ is a meromorphic function, then we have
	\begin{align}
		\label{C14a}
		\begin{aligned}
			m\left(r,\frac{\L_\W f}{f}\right) &\le m\left(r,\frac{\D_\W\mu(x;a+\frac{1}{2},b+\frac{1}{2},c+\frac{1}{2},d+\frac{1}{2})\D_\W f}{\mu(x;a+\frac{1}{2},b+\frac{1}{2},c+\frac{1}{2},d+\frac{1}{2})\D_\W f}\right) \\
			&\ \ \ + m\left(r,\frac{\D_\W f}{f}\right) + m\left(r,\frac{\mu(x;a+\frac{1}{2},b+\frac{1}{2},c+\frac{1}{2},d+\frac{1}{2})}{\mu(x;a,b,c,d)}\right).
		\end{aligned}
	\end{align}
	Now for any fixed $\phi\in(-\pi,\pi)$ and any $\alpha,\beta\in\mathbb{C}$, we have
	\[
		\frac{\Gamma(\rho e^{i\phi}+\alpha)}{\Gamma(\rho e^{i\phi}+\beta)} = (\rho e^{i\phi})^{\alpha-\beta}\left[1+\frac{(\alpha-\beta)(\alpha-\beta-1)}{2\rho e^{i\phi}}+O\left(\frac{1}{\rho^2}\right)\right]
	\]
	as $\rho\to\infty$ \cite{Tricomi-Erdelyi}, so given any $a\in\mathbb{C}$ and any fixed $\theta\in(-\pi,\pi)$, by taking $\alpha=a+\frac{1}{2}$, $\beta=a$ and $\phi=\frac{\theta+\pi}{2},\frac{\theta-\pi}{2}$, we have
	\begin{align*}
		&\ \ \ \,\ln^+\left|\frac{\Gamma(a+\frac{1}{2}+i\sqrt{r}e^{i\frac{\theta}{2}})\Gamma(a+\frac{1}{2}-i\sqrt{r}e^{i\frac{\theta}{2}})}{\Gamma(a+i\sqrt{r}e^{i\frac{\theta}{2}})\Gamma(a-i\sqrt{r}e^{i\frac{\theta}{2}})}\right| \\
		&= \ln^+\left|r^{\frac{1}{4}}\left(1+\frac{a-\frac{1}{4}}{2i\sqrt{r}e^{i\frac{\theta}{2}}}+O\left(\frac{1}{r}\right)\right)r^{\frac{1}{4}}\left(1-\frac{a-\frac{1}{4}}{2i\sqrt{r}e^{i\frac{\theta}{2}}}+O\left(\frac{1}{r}\right)\right)\right| \\
		&= \ln^+\left|\sqrt{r}\left(1+O\left(\frac{1}{r}\right)\right)\right| \sim \frac{1}{2}\ln r
	\end{align*}
	as $r\to\infty$.\ \ Lebesgue dominated convergence theorem then gives
	\[
		\lim_{r\to\infty}{\frac{1}{2\pi}\int_{-\pi}^{\pi}{\frac{1}{\ln r}\ln^+\left|\frac{\Gamma(a+\frac{1}{2}+i\sqrt{r}e^{i\frac{\theta}{2}})\Gamma(a+\frac{1}{2}-i\sqrt{r}e^{i\frac{\theta}{2}})}{\Gamma(a+i\sqrt{r}e^{i\frac{\theta}{2}})\Gamma(a-i\sqrt{r}e^{i\frac{\theta}{2}})}\right|\,d\theta}}=\frac{1}{2}
	\]
	for any $a\in\mathbb{C}$.\ \ This implies that $m\left(|x|=r,\frac{\Gamma(a+\frac{1}{2}+iz)\Gamma(a+\frac{1}{2}-iz)}{\Gamma(a+iz)\Gamma(a-iz)}\right) \sim \frac{1}{2}\ln r$, and replacing $a$ by $b,c,d$ successively, we obtain
	\[
		m\left(r,\frac{\mu(x;a+\frac{1}{2},b+\frac{1}{2},c+\frac{1}{2},d+\frac{1}{2})}{\mu(x;a,b,c,d)}\right) \sim 2\ln r
	\]
	as $r\to\infty$.\ \ On the other hand, since $f$ is of finite order $\sigma(f)=\sigma$, the order of $\mu(x;a+\frac{1}{2},b+\frac{1}{2},c+\frac{1}{2},d+\frac{1}{2})\D_\W f$ is at most
	\begin{align*}
		&\ \ \ \,\sigma\left(\mu\left(x;a+\frac{1}{2},b+\frac{1}{2},c+\frac{1}{2},d+\frac{1}{2}\right)\D_\W f\right) \\
		&\le \max\left\{\sigma\left(\mu\left(x;a+\frac{1}{2},b+\frac{1}{2},c+\frac{1}{2},d+\frac{1}{2}\right)\right),\sigma(\D_\W f)\right\} \\
		&\le \max\left\{\frac{1}{2},\sigma(f)\right\}.
	\end{align*}
	Thus applying Theorem~\ref{WilsonLogLemma} to \eqref{C14a}, we have for every $\varepsilon>0$,
	\begin{align*}
		m\left(r,\frac{\L_\W f}{f}\right) &\le O(r^{\sigma(\mu(x;a+\frac{1}{2},b+\frac{1}{2},c+\frac{1}{2},d+\frac{1}{2})\D_\W f)-\frac{1}{2}+\varepsilon}) + O(r^{\sigma(f)-\frac{1}{2}+\varepsilon}) + 2\ln r \\
		&\le O(r^{\max\{\frac{1}{2},\sigma\}-\frac{1}{2}+\varepsilon}) + O(r^{\sigma-\frac{1}{2}+\varepsilon}) + 2\ln r = O(r^{\max\{0,\sigma-\frac{1}{2}\}+\varepsilon})
	\end{align*}
	as $r\to\infty$.
\end{proof}

From the above lemma, we deduce the following theorem about meromorphic solutions to equations of the form $\L_\W (y) = \phi y$.

\begin{theorem}
Let $\phi$ be an entire function of order $\sigma_\phi>0$.\ \ Then any non-trivial meromorphic solution $f$ to the equation $\L_\W (y) = \phi y$ must have order $\sigma\ge\sigma_\phi+\frac{1}{2}$.
\end{theorem}
\begin{proof}
	Suppose that $f$ is a meromorphic solution to the equation $\L_\W (y) = \phi y$ having order $\sigma<\sigma_\phi+\frac{1}{2}$.\ \ Then we take $\varepsilon:=\sigma_\phi-\max\{0,\sigma-\frac{1}{2}\}>0$.\ \ Since $\phi$ is entire, we have
	\begin{align*}
		T(r,\phi) &= m(r,\phi) = m\left(r,\frac{\L_\W f}{f}\right) = O(r^{\max\{0,\sigma-\frac{1}{2}\}+\frac{\varepsilon}{2}}) = O(r^{\sigma_\phi-\frac{\varepsilon}{2}})
	\end{align*}
	as $r\to\infty$, which is a contradiction.
\end{proof}

\subsection{Wilson analogue of Clunie's lemma}

We can derive from Theorem~\ref{WilsonLogLemma} a Wilson analogue of Clunie's Lemma, which is a result about some non-linear Wilson difference equations.\ \ Before stating the result we need a definition.

\begin{definition}
Let $f$ be a complex function.\ \ A \textbf{\textit{Wilson difference polynomial in $\boldsymbol{f}$}} is a finite sum of finite products of the form
\[
	P(f) = \sum_j{P_j\prod_l{(\D_\W^l f)^{d_{l,j}}}},
\]
where $d_{l,j}$ are non-negative integers and the coefficients $P_j$ are polynomials in one variable.\ \ The \textbf{\textit{degree}} of $P$ \textbf{\textit{over $\boldsymbol{f}$}} is the non-negative integer
\[
	\deg_f{P} := \max_j{\sum_l{d_{l,j}}}.
\]
\end{definition}

\begin{theorem}
\label{WClunie}
\textup{(Clunie's Lemma for the Wilson operator)}
Let $n$ be a positive integer and $f\not\equiv 0$ be a meromorphic solution of finite order $\sigma$ of the Wilson difference equation
\begin{align}
	\label{C15} y^n P(y) = Q(y),
\end{align}
where $P(y)$ and $Q(y)$ are Wilson difference polynomials in $y$ and $\deg_f{Q} \le n$.\ \ Then for every $\varepsilon>0$, we have
\begin{align}
	\label{C16} m(r,P) = O(r^{\sigma-\frac{1}{2}+\varepsilon}) + O(\ln r)
\end{align}
as $r\to\infty$.
\end{theorem}
The proof of Theorem~\ref{WClunie} is a standard application of Theorem~\ref{WilsonLogLemma}, and is therefore omitted.\ \ We note that the analogues of Clunie's lemma for the usual difference operator $\Delta$ and the Askey-Wilson divided-difference operator $\D_q$ are given in \cite{HK-1}, \cite{Laine-Yang} and \cite{Chiang-Feng3} respectively.

\subsection{Wilson interpolation equations}
Next we look at the growth of meromorphic solutions to some Wilson interpolation equations, which can be deduced using Theorem~\ref{WilsonLogLemma}.

\begin{theorem}
\label{WIE1}
Let $n$ be a positive integer, $A_0,A_1,\ldots,A_n$ be entire functions with orders $\sigma_0,\sigma_1,\ldots,\sigma_n$ respectively, such that $\displaystyle \sigma_l>\max_{k\in\{0,1,\ldots,n\}\setminus\{l\}}{\sigma_k}$ for some $l\in\{0,1,\ldots,n\}$, and $f$ be a non-trivial meromorphic solution to the Wilson interpolation equation
\begin{align}
	\label{C20} A_n(x)y(x^{+(n)}) + A_{n-1}(x)y(x^{+(n-1)}) + \cdots + A_1(x)y(x^+) + A_0(x)y(x) = 0.
\end{align}
Then $f$ has order $\sigma\ge\sigma_l+\frac{1}{2}$.
\end{theorem}
\begin{proof}
	We suppose the contrary that $f$ is a non-trivial meromorphic solution to \eqref{C20} having order $\sigma<\sigma_l+\frac{1}{2}$.\ \ Now $f\not\equiv 0$, so we may substitute $f$ into \eqref{C20} and divide both sides by $f(x^{+(l)})$ to get
	\begin{align}
		\label{C21} A_n(x)\frac{f(x^{+(n)})}{f(x^{+(l)})} + \cdots + A_l(x) + \cdots + A_1(x)\frac{f(x^{+})}{f(x^{+(l)})} + A_0(x)\frac{f(x)}{f(x^{+(l)})} = 0.
	\end{align}
	Since $\displaystyle \sigma_l>\max_{k\in\{0,1,\ldots,n\}\setminus\{l\}}{\sigma_k}$, we may take $\displaystyle\max_{k\in\{0,1,\ldots,n\}\setminus\{l\}}{\sigma_k}<s<\sigma_l$ and take $\varepsilon>0$ such that we simultaneously have
	\[
		\sigma+2\varepsilon<\sigma_l+\frac{1}{2} \hspace{25px} \mbox{and} \hspace{25px} s+2\varepsilon<\sigma_l.
	\]
	Then $T(r,f)=O(r^{\sigma+\frac{\varepsilon}{2}})$ as $r\to\infty$.\ \ Thus by choosing $\alpha=1-\varepsilon$ in \eqref{B12} and \eqref{B14}, it is easy to see that for each $k\in\{0,1,\ldots,n\}\setminus\{l\}$, we have
	\[
		m\left(r,\frac{f(x^{+(k)})}{f(x^{+(l)})}\right) = O(r^{\sigma-\frac{1}{2}+\varepsilon})
	\]
	as $r\to\infty$.\ \ Applying this to \eqref{C21}, we obtain
	\begin{align*}
		m(r,A_l) &\le \sum_{k\in\{0,1,\ldots,n\}\setminus\{l\}}{\left[m\left(r,\frac{f(x^{+(k)})}{f(x^{+(l)})}\right) + m(r,A_k)\right]} + \ln n \\
		&= O(r^{\sigma-\frac{1}{2}+\varepsilon}) + o(r^{s+\varepsilon}) \\
		&= o(r^{\sigma_l-\varepsilon})
	\end{align*}
	as $r\to\infty$, which is an obvious contradiction.
\end{proof}

It is clear that the signs of the shifts are not important, and Theorem~\ref{WIE1} still holds even if we include terms of the form $A_{-l}(x)y(x^{-(l)})$ in \eqref{C20}.\ \ To illustrate that the equality in $\sigma\ge\sigma_l+\frac{1}{2}$ can actually be achieved, we consider the following example taken from Ruijsenaars' work in \cite{Ruijsenaars}.

\begin{example}
\label{Ghyp}
In this example, we consider the Wilson interpolation equation
\begin{align}
	\label{C22} y(x^+) = 2\cosh(\pi\sqrt{x}) y(x^-).
\end{align}
According to Ruijsenaars' article \cite{Ruijsenaars}, a solution to the difference equation
\begin{align}
	\label{C23} y\left(z+\frac{ia}{2}\right) = 2\cosh{\frac{\pi z}{b}}\,y\left(z-\frac{ia}{2}\right),
\end{align}
where $a,b>0$, is given by the hyperbolic gamma function
\[
	G_\mathrm{hyp}(a,b;z) := e^{i\int_0^\infty{\left(\frac{\sin{2tz}}{2\sinh{at}\sinh{bt}}-\frac{z}{abt}\right)\frac{dt}{t}}}
\]
for every $z\in\mathbb{C}$ satisfying $|\mathfrak{I}2z|<a+b$, which can be extended to a meromorphic function on the whole complex plane.\ \ Noting that the solution set of the difference equation \eqref{C23} forms a vector space, and that
\[
	\displaystyle G_\mathrm{hyp}(a,b;-z)=\frac{1}{G_\mathrm{hyp}(a,b;z)}
\]
for every $z\in\mathbb{C}$, we can observe that a meromorphic solution to the Wilson interpolation equation \eqref{C22} is given by
\[
	f(x) = \frac{G_\mathrm{hyp}(1,1;\sqrt{x})+G_\mathrm{hyp}(1,1;-\sqrt{x})}{2} = \cos\left[\int_0^\infty{\left(\frac{\sin{2t\sqrt{x}}}{2\sinh^2{t}}-\frac{\sqrt{x}}{t}\right)\frac{dt}{t}}\right].
\]
Now since $G_\mathrm{hyp}(1,1;x)$ is of order $2$, we see that the order of this solution $f$ is at most $1$.\ \ $f$ has poles at $-k^2$ of order $k$ for each $k\in\mathbb{N}$.\ \ On the other hand, it is obvious that the entire function $2\cosh(\pi\sqrt{x})$ has order $\frac{1}{2}$.\ \ Therefore $f$ must be of order exactly $1$, and the equality in the conclusion of Theorem~\ref{WIE1} is achieved.
\end{example}

In addition to Theorem~\ref{WClunie} and Theorem~\ref{WIE1}, more results about Wilson difference equations and Wilson interpolation equations will be established after we give a pointwise estimate of the logarithmic Wilson difference in the next section.

\section{Pointwise Estimate of the Logarithmic Wilson Difference}
\label{sec:Ptwise}

The lemma on logarithmic Wilson difference we have obtained in \S\ref{sec:WilsonLogLemma} is about the proximity function $m\left(r,\frac{\D_\W f}{f}\right)$, which is a measure on the overall growth of the logarithmic Wilson difference $\frac{\D_\W f}{f}$.\ \ In this section, we will look at the pointwise growth of $\frac{\D_\W f}{f}$, which will play a very crucial role in Wilson difference equations.

\subsection{The pointwise estimate}

Our pointwise estimate of the growth of $\frac{\D_\W f}{f}$ depends on the following deep and important result by H. Cartan, which can be found in his paper \cite{Cartan1} published in 1928.

\begin{lemma}
\label{CartanLemma}
\textup{(Cartan's Lemma)}
	Given any $p$ complex numbers $x_1,x_2,\ldots,x_p$ and any real number $B>0$, there exist finitely many closed disks $D_1,D_2,\ldots,D_q$ in the complex plane, with radii $r_1,r_2,\ldots,r_q$ respectively, such that
	\begin{enumerate}[(i)]
		\item $r_1+r_2+\cdots+r_q=2B$, and
		\item For every complex number $x$ outside all of the disks $D_j$, there exists a permutation of the $p$ given points, say $\hat{x}_1,\hat{x}_2,\ldots,\hat{x}_p$ (which may depend on $x$), that satisfies
		\begin{align*}
			|x-\hat{x}_\lambda| > B\frac{\lambda}{p}, \hspace{30px} \mbox{for $\lambda=1,2,\ldots,p$.}
		\end{align*}
	\end{enumerate}
\end{lemma}

In what follows, we will denote by
\[
	n(r) := n(r,f)+n\left(r,\frac{1}{f}\right)
\]
the unintegrated counting function for both zeros and poles of $f$.\ \ Our main result in this section starts from the following lemma.
\begin{lemma}
\label{Estimate6}
	Let $f\not\equiv 0$ be a meromorphic function.\ \ Then for every $\gamma>1$, there exist a subset $E\subset(1,\infty)$ of finite logarithmic measure and a constant $A_\gamma$ depending only on $\gamma$, such that for every complex number $x$ with modulus $r\notin E\cup[0,1]$, we have
	\begin{align}
		\label{F1} \left|\ln\left|\frac{f(x^+)}{f(x)}\right|\right| \le A_\gamma\left\{\frac{T(\gamma r,f)}{\sqrt{r}} + \frac{n(\gamma r)}{\sqrt{r}}\ln^\gamma{r} \ln^+ [n(\gamma r)]\right\}.
	\end{align}
	In particular, if $f$ has finite order $\sigma$, then for every $\varepsilon>0$, there exists a subset $E\subset(1,\infty)$ of finite logarithmic measure such that for every complex number $x$ with modulus $r\notin E\cup[0,1]$, we have
	\[
		e^{-r^{\sigma-\frac{1}{2}+\varepsilon}} \le \left|\frac{f(x^+)}{f(x)}\right| \le e^{r^{\sigma-\frac{1}{2}+\varepsilon}}.
	\]
\end{lemma}
\begin{proof}
	For every complex number $x$, we write $|x|=r$ and let $R$ be a positive number such that $R>\frac{1}{4}$ and $r<(\sqrt{R}-\frac{1}{2})^2$.\ \ Then we use the Poisson-Jensen formula as in \S\ref{sec:WilsonLogLemma} to obtain \eqref{B5}.\ \ Next we let $\alpha>1$ be arbitrary and let $R=(\sqrt{\alpha r}+\frac{1}{2})^2$.\ \ For sufficiently large $r$, namely for $\displaystyle r>\frac{1}{4(\alpha-\sqrt{\alpha})^2}$, we have $\sqrt{\alpha r}<\sqrt{\alpha r}+\frac{1}{2}<\sqrt{\alpha^2 r}$.\ \ Thus applying Lemma~\ref{Estimate1} and \eqref{B13} to \eqref{B5}, we have
	\begin{align}
		\label{F2}
		{\small\begin{aligned}
			&\ \ \ \,\left|\ln\left|\frac{f(x^+)}{f(x)}\right|\right| \\
			&\le \frac{2(\sqrt{\alpha r}+\frac{1}{2})^2(2\sqrt{r}+\frac{1}{2})}{[(\sqrt{\alpha r}+\frac{1}{2})^2-r][(\sqrt{\alpha r}+\frac{1}{2})^2-(\sqrt{r}+\frac{1}{2})^2]}\left[T\left(\left(\sqrt{\alpha r}+\frac{1}{2}\right)^2,f\right)+\ln^+\frac{1}{|f(0)|}\right] \\
			&\ \ \ \,+\frac{2\sqrt{r}+\frac{1}{2}}{(\sqrt{\alpha r}+\frac{1}{2})^2-(\sqrt{r}+\frac{1}{2})^2}\,n\left(\left(\sqrt{\alpha r}+\frac{1}{2}\right)^2\right) + \left(\sqrt{r}+\frac{1}{4}\right)\sum_{|c_\lambda|<\alpha^2 r}{\frac{1}{|z^2-c_\lambda|}} \\
			&\le \frac{\alpha^2 (4\sqrt{r}+1)}{(\alpha-1)[(\alpha-1)r-\sqrt{r}-\frac{1}{4}]}\left[T(\alpha^2 r,f)+\ln^+\frac{1}{|f(0)|}\right] \\
			&\ \ \ \,+ \frac{2\sqrt{r}+\frac{1}{2}}{(\alpha-1)r-\sqrt{r}-\frac{1}{4}}\,n(\alpha^2 r) + \left(\sqrt{r}+\frac{1}{4}\right)\sum_{|c_\lambda|<\alpha^2 r}{\frac{1}{|x-c_\lambda|}},
		\end{aligned}}
	\end{align}
	where the sequence $\{c_\lambda\}_\lambda$ is the union of the four sequences $\{a_\nu\}_\nu$, $\{b_\mu\}_\mu$, $\{a_\nu^-\}_\nu$ and $\{b_\mu^-\}_\mu$, and is ordered by increasing modulus. \\

	Next, we rename the arbitrary number $\alpha^2>1$ by the arbitrary number $\beta>1$, and estimate the series on the right-hand side of \eqref{F2} as in \cite[(7.6)-(7.9)]{Gundersen}.\ \ We only need to consider those complex numbers $x$ with modulus $r>1$, so we let
	\begin{align}
		\label{F3} \beta^p\le r\le\beta^{p+1},
	\end{align}
	where $p$ is a positive integer sufficiently large so that
	\begin{align}
		\label{F3A} n(\beta^{p+2})\ge e^{\frac{1}{\beta-1}}.
	\end{align}
	Then applying Lemma~\ref{CartanLemma} to the points $c_1,c_2,\ldots,c_{n(\beta^{p+2})}$ with $\displaystyle B=\frac{\beta^p}{\ln^\beta(\beta^p)}$, we deduce that there exist finitely many closed disks $D_1,D_2,\ldots,D_q$ with sum of radii $2B$, such that for every $x$ outside all of these disks, there is a permutation of the points, say $d_1,d_2,\ldots,d_{n(\beta^{p+2})}$, that satisfies the inequality
	\begin{align}
		\label{F4} |x-d_\lambda| > \lambda\frac{B}{n(\beta^{p+2})} = \lambda\frac{\beta^p}{n(\beta^{p+2})\ln^\beta(\beta^p)}
	\end{align}
	for every $\lambda=1,2,\ldots,n(\beta^{p+2})$.\ \ \eqref{F3}, \eqref{F3A} and \eqref{F4} imply that for every $x$ with modulus $\beta^p\le r\le\beta^{p+1}$ and is located outside all of the disks $D_j$, we have
	\begin{align*}
		\sum_{|c_\lambda|<\beta r}{\frac{1}{|x-c_\lambda|}} &\le \sum_{\lambda=1}^{n(\beta^{p+2})}{\frac{1}{|x-d_\lambda|}} < \frac{n(\beta^{p+2})\ln^\beta(\beta^p)}{\beta^p}\sum_{\lambda=1}^{n(\beta^{p+2})}{\frac{1}{\lambda}} \\
		&\le \beta r^{-1}n(\beta^{p+2})\ln^\beta r\{1+\ln^+ [n(\beta^{p+2})]\} \\
		&\le \beta r^{-1}n(\beta^2 r)\ln^\beta r\{1+\ln^+ [n(\beta^2 r)]\} \\
		&\le \beta^2\frac{n(\beta^2 r)}{r} \ln^\beta{r} \ln^+ [n(\beta^2 r)].
	\end{align*}
	Therefore for every such $x$, we have
	\begin{align*}
		\left|\ln\left|\frac{f(x^+)}{f(x)}\right|\right| \le &\;\frac{\beta(4\sqrt{r}+1)}{(\sqrt{\beta}-1)[(\sqrt{\beta}-1)r-\sqrt{r}-\frac{1}{4}]}\left[T(\beta r,f) + \ln^+\frac{1}{|f(0)|}\right] \\
		&+ \frac{2\sqrt{r}+\frac{1}{2}}{(\sqrt{\beta}-1)r-\sqrt{r}-\frac{1}{4}}\,n(\beta r) \\
		&+\left(\sqrt{r}+\frac{1}{4}\right)\beta^2\frac{n(\beta^2 r)}{r} \ln^\beta{r} \ln^+ [n(\beta^2 r)] \\
		\le &\;A_\beta\left\{\frac{T(\beta r,f)}{\sqrt{r}} + \frac{n(\beta^2 r)}{\sqrt{r}}\ln^\beta{r} \ln^+ [n(\beta^2 r)]\right\},
	\end{align*}
	where the constant $A_\beta$ depends on $\beta$ only.\ \ Since the functions $T(r,f)$ and $\ln r$ are increasing, we can actually write
	\begin{align*}
		\left|\ln\left|\frac{f(x^+)}{f(x)}\right|\right| & \le A_\beta\left\{\frac{T(\beta^2 r,f)}{\sqrt{r}} + \frac{n(\beta^2 r)}{\sqrt{r}}\ln^{\beta^2}{r} \ln^+ [n(\beta^2 r)]\right\} \\
		&= A_\gamma\left\{\frac{T(\gamma r,f)}{\sqrt{r}} + \frac{n(\gamma r)}{\sqrt{r}}\ln^\gamma{r} \ln^+ [n(\gamma r)]\right\}
	\end{align*}
	by renaming $\beta^2>1$ as $\gamma>1$, where $A_\gamma$ is a constant depending on $\gamma$ only. \\

	It now remains to show that the exceptional set $E\subset(1,\infty)$ is of finite logarithmic measure.\ \ For every positive integer $p$ sufficiently large so that \eqref{F3A} holds, say $p\ge p_0\ge 1$, the sum of the diameters of the exceptional disks is $\displaystyle 2\cdot2B=\frac{4\beta^p}{\ln^\beta(\beta^p)}$.\ \ Now we can write the exceptional set as
	\[
		E=\bigcup_{p=p_0}^\infty{E_p},
	\]
	where $E_p$ is the intersection of the closed interval $[\beta^p,\beta^{p+1}]$ and the finite union of all the annuli generated by revolving the disks $D_j$ about the origin.\ \ Since $E_p$ has Lebesgue measure at most $\displaystyle\frac{4\beta^p}{\ln^\beta(\beta^p)}$, it follows that
	\begin{align*}
		\int_E{\frac{dr}{r}} &= \sum_{p=p_0}^\infty{\int_{E_p}{\frac{dr}{r}}} \le \sum_{p=p_0}^\infty{\int_{\left[\beta^p,\beta^p+\frac{4\beta^p}{\ln^\beta(\beta^p)}\right]}{\frac{dr}{r}}} \\
		&= \sum_{p=p_0}^\infty{\left\{\ln\left[\beta^p+\frac{4\beta^p}{\ln^\beta(\beta^p)}\right] - \ln(\beta^p)\right\}} \\
		&\le \sum_{p=p_0}^\infty{\frac{4}{\ln^\beta(\beta^p)}} < \infty,
	\end{align*}
	and so $E$ has finite logarithmic measure. \\
	
	Finally, we suppose in particular that $f$ has finite order $\sigma$.\ \ \eqref{F1} implies that for every $\gamma>1$ and $\varepsilon>0$, there exists an exceptional set $E$ of finite logarithmic measure such that
	\begin{align*}
		\left|\ln\left|\frac{f(x^+)}{f(x)}\right|\right| &= O\left(\frac{T(\gamma r,f)}{\sqrt{r}} + \frac{n(\gamma r)}{\sqrt{r}}\ln^\gamma{r} \ln^+ [n(\gamma r)]\right) \\
		&= O(r^{\sigma-\frac{1}{2}+\frac{\varepsilon}{2}} + (\gamma r)^{\sigma-\frac{1}{2}}\ln^\gamma{r}\ln^+ (\gamma r)^\sigma) \\
		&= O(r^{\sigma-\frac{1}{2}+\frac{\varepsilon}{2}})
	\end{align*}
	as $r\to\infty$ outside $E$, where the second equality follows from the fact that the exponents of convergence of zeros and poles of $f$ are lower than or equal to $\sigma$.\ \ This gives $\left|\ln\left|\frac{f(x^+)}{f(x)}\right|\right|\le r^{\sigma-\frac{1}{2}+\varepsilon}$ for every $x$ with modulus $r\notin E\cup[0,1]$, and so we have
	\[
		e^{-r^{\sigma-\frac{1}{2}+\varepsilon}} \le \left|\frac{f(x^+)}{f(x)}\right| \le e^{r^{\sigma-\frac{1}{2}+\varepsilon}}
	\]
	for every $x$ with modulus $r\notin E\cup[0,1]$.
\end{proof}

An exceptional set of radii $r$ of finite logarithmic measure can sometimes be of infinite Lebesgue measure, but it becomes sparser and sparser as $r\to\infty$. \\

In the above lemma, we generated an annular exceptional set by revolving the exceptional disks $D_j$ about the origin.\ \ In fact, the exceptional set can also be generated in a radial manner.\ \ This gives us the following lemma, in which the exceptional set has zero Lebesgue measure instead of finite logarithmic measure.

\begin{lemma}
\label{Estimate7}
	Let $f\not\equiv 0$ be a meromorphic function.\ \ Then for every $\gamma>1$, there exists a constant $B_\gamma$ depending only on $\gamma$, such that for almost every $\theta\in [-\pi,\pi)$, we have
	\[
		\left|\ln\left|\frac{f((re^{i\theta})^+)}{f(re^{i\theta})}\right|\right| \le B_\gamma\left\{\frac{T(\gamma r,f)}{\sqrt{r}} + \frac{n(\gamma r)}{\sqrt{r}}\ln^\gamma{r} \ln^+ [n(\gamma r)]\right\}
	\]
	for every sufficiently large $r$.\ \ In particular, if $f$ has finite order $\sigma$, then for every $\varepsilon>0$ and almost every $\theta\in [-\pi,\pi)$, we have
	\[
		e^{-r^{\sigma-\frac{1}{2}+\varepsilon}} \le \left|\frac{f((re^{i\theta})^+)}{f(re^{i\theta})}\right| \le e^{r^{\sigma-\frac{1}{2}+\varepsilon}}
	\]
	for every sufficiently large $r$.
\end{lemma}

Let us put Lemma~\ref{Estimate6} into a more useful form.
\begin{theorem}
\label{PtwiseEst}
\textup{(Pointwise estimate of the logarithmic Wilson difference, radial version)}
	Let $f\not\equiv 0$ be a meromorphic function.\ \ Then for every $\gamma>1$, there exist a subset $E_1\subset(1,\infty)$ of finite logarithmic measure and a constant $A_\gamma$ depending only on $\gamma$, such that for every complex number $x$ with modulus $r\notin E_1\cup[0,1]$, we have
	\begin{align}
		\label{F5} \left|\frac{(\D_\W f)(x)}{f(x)}\right| \le 2e^{A_\gamma\left\{\frac{T(\gamma r,f)}{\sqrt{r}} + \frac{n(\gamma r)}{\sqrt{r}}\ln^\gamma{r} \ln^+ [n(\gamma r)]\right\}}.
	\end{align}
	In particular, if $f$ is of finite order $\sigma$, then for every positive integer $k$ and every $\varepsilon>0$, there exists a subset $E\subset(1,\infty)$ of finite logarithmic measure such that for every complex number $x$ with modulus $r\notin E\cup[0,1]$, we have
	\[
		\left|\frac{(\D_\W^k f)(x)}{f(x)}\right| \le e^{kr^{\sigma-\frac{1}{2}+\varepsilon}}.
	\]
\end{theorem}
\begin{proof}
	Lemma~\ref{Estimate6} implies that \eqref{F5} holds for every $x$ with modulus $r$ outside the union of two exceptional sets of finite logarithmic measure, which is still of finite logarithmic measure.
	If $f$ is of finite order $\sigma$.\ \ Then \eqref{F5} implies that for every $\varepsilon>0$,
	\begin{align}
		\label{F6} \left|\frac{(\D_\W f)(x)}{f(x)}\right| \le \frac{e^{r^{\sigma-\frac{1}{2}+\frac{\varepsilon}{2}}}}{\sqrt{r}} \le e^{r^{\sigma-\frac{1}{2}+\varepsilon}},
	\end{align}
	which holds for every $x$ with modulus $r$ outside an exceptional set of finite logarithmic measure.\ \ Since Lemma~\ref{Estimate3} implies that $\D_\W^l f$ has order at most $\sigma$ for every positive integer $l$, applying \eqref{F6} $k$ times, we have
	\[
		\left|\frac{(\D_\W^k f)(x)}{f(x)}\right| \le e^{kr^{\sigma-\frac{1}{2}+\varepsilon}}
	\]
	for every $x$ with modulus $r$ outside the union of $k$ exceptional sets of finite logarithmic measure, which is still of finite logarithmic measure.
\end{proof}

Similarly, Lemma~\ref{Estimate7} gives the following theorem.

\begin{theorem}
\textup{(Pointwise estimate of the logarithmic Wilson difference, angular version)}
	Let $f\not\equiv 0$ be a meromorphic function.\ \ Then for every $\gamma>1$, there exists a constant $B_\gamma$ depending only on $\gamma$, such that for almost every $\theta\in [-\pi,\pi)$, we have
	\[
		\left|\frac{(\D_\W f)(re^{i\theta})}{f(re^{i\theta})}\right| \le 2e^{B_\gamma\left\{\frac{T(\gamma r,f)}{\sqrt{r}} + \frac{n(\gamma r)}{\sqrt{r}}\ln^\gamma{r} \ln^+ [n(\gamma r)]\right\}}
	\]
	for every sufficiently large $r$.\ \ In particular, if $f$ is of finite order $\sigma$, then for every positive integer $k$, every $\varepsilon>0$ and almost every $\theta\in [-\pi,\pi)$, we have
	\[
		\left|\frac{(\D_\W^k f)(re^{i\theta})}{f(re^{i\theta})}\right| \le e^{kr^{\sigma-\frac{1}{2}+\varepsilon}}
	\]
	for every sufficiently large $r$.
\end{theorem}

\subsection{Applications of the pointwise estimate in Wilson difference equations and Wilson interpolation equations}

The pointwise estimate of the logarithmic Wilson difference plays a crucial role in Wilson difference equations.\ \ A typical application as in \cite{Chiang-Feng2} yields the following result about the growth of meromorphic solutions to a linear Wilson difference equation with {polynomial} coefficients.

\begin{theorem}
\label{WDE}
Let $n$ be a positive integer, $p_0,p_1,\ldots,p_n$ be {polynomials} such that {$\deg p_0>\deg p_k$} for each $k\in\{1,2,\ldots,n\}$, and $f$ be a non-trivial meromorphic solution to the Wilson difference equation
\begin{align}
	\label{F7} p_n\D_\W^n y + p_{n-1}\D_\W^{n-1} y + \cdots + p_1\D_\W y + p_0y = 0.
\end{align}
Then $f$ has order $\sigma\ge\frac{1}{2}$.
\end{theorem}
\begin{proof}
	We suppose the contrary that $f$ is a non-trivial meromorphic solution to \eqref{F7} having order $\sigma<\frac{1}{2}$.\ \ Now $p_0f\not\equiv 0$, so we may substitute $f$ into \eqref{F7} and divide both sides by $p_0f$ to get
	\begin{align}
		\label{F8} \frac{p_n}{p_0}\frac{\D_\W^n f}{f} + \frac{p_{n-1}}{p_0}\frac{\D_\W^{n-1} f}{f} + \cdots + \frac{p_1}{p_0}\frac{\D_\W f}{f} = -1.
	\end{align}
	Now for each $k\in\{1,2,\ldots,n\}$, writing $|x|=r$, we have
	\[
		\left|\frac{p_k(x)}{p_0(x)}\right| = {o(1)}
	\]
	as $r\to\infty$, and we have
	\[
		\left|\frac{(\D_\W^k f)(x)}{f(x)}\right| = {O(1)}
	\]
	as $r\to\infty$ outside an exceptional set $E_k\subset(1,\infty)$ of finite logarithmic measure, by applying Theorem~\ref{PtwiseEst} to $f$.\ \ Therefore the value taken by the left-hand side of \eqref{F8} tends to zero as $j\to\infty$ on every strictly increasing sequence $\{r_j\}_{j=1}^\infty$ outside the exceptional set $\displaystyle E=\bigcup_{k=1}^n{E_k}\subset(1,\infty)$ of finite logarithmic measure, which is a contradiction.
\end{proof}

Here is a result about the growth of meromorphic solutions to some Wilson interpolation equations {with polynomial coefficients}, which can be deduced directly using Lemma~\ref{Estimate6}.

\begin{theorem}
\label{WIE2}
Let $n$ be a positive integer, $p_0,p_1,\ldots,p_n$ be {polynomials} such that for some $l\in\{0,1,\ldots,n\}$ we have {$\deg p_l>\deg p_k$} for every $k\in\{0,1,\ldots,n\}\setminus\{l\}$, and $f$ be a non-trivial meromorphic solution to the Wilson interpolation equation
\begin{align}
	\label{F9} p_n(x)y(x^{+(n)}) + p_{n-1}(x)y(x^{+(n-1)}) + \cdots + p_1(x)y(x^+) + p_0(x)y(x) = 0.
\end{align}
Then $f$ has order $\sigma\ge\frac{1}{2}$.
\end{theorem}
\begin{proof}
	We suppose on the contrary that $f$ is a non-trivial meromorphic solution to \eqref{F9} having order $\sigma<\frac{1}{2}$.\ \ Now we have $p_l\not\equiv 0$ and $f\not\equiv 0$, so we may substitute $f$ into \eqref{F9} and divide both sides by $p_l(x)f(x^{+(l)})$ to get
	\begin{align}
		\label{F10} \frac{p_n(x)}{p_l(x)}\frac{f(x^{+(n)})}{f(x^{+(l)})} + \cdots + \frac{p_l(x)}{p_l(x)} + \cdots + \frac{p_1(x)}{p_l(x)}\frac{f(x^{+})}{f(x^{+(l)})} + \frac{p_0(x)}{p_l(x)}\frac{f(x)}{f(x^{+(l)})} = 0.
	\end{align}
	Since $\sigma<\frac{1}{2}$, we take $\varepsilon>0$ such that $\varepsilon<\frac{1}{2}-\sigma$.\ \ If we write $|x|=r$, then Lemma~\ref{Estimate6} implies that for each $k\in\{0,1,\ldots,n\}\setminus\{l\}$, we have
	\begin{align*}
		\left|\frac{f(x^{+(k)})}{f(x^{+(l)})}\right| &\le e^{r^{\sigma-\frac{1}{2}+\varepsilon}} = O(1)
	\end{align*}
	as $r\to\infty$ outside an exceptional set $E_k\subset(1,\infty)$ of finite logarithmic measure.\ \ Applying this to \eqref{F10}, we have
	\begin{align*}
		1 = \left|\frac{p_l(x)}{p_l(x)}\right| \le \sum_{k\in\{0,1,\ldots,n\}\setminus\{l\}}{\left|\frac{p_k(x)}{p_l(x)}\right|\left|\frac{f(x^{+(k)})}{f(x^{+(l)})}\right|} \le O(1)\sum_{k\in\{0,1,\ldots,n\}\setminus\{l\}}{\left|\frac{p_k(x)}{p_l(x)}\right|}
	\end{align*}
	as $r\to\infty$ outside the exceptional set $\displaystyle E=\bigcup_{k=1}^n{E_k}\subset(1,\infty)$ of finite logarithmic measure, which is a contradiction to the assumption that {$\deg p_l>\deg p_k$} for every $k\in\{0,1,\ldots,n\}\setminus\{l\}$.
\end{proof}

\section{Discussion}

The Wilson operator $\D_\W$ introduced in the previous sections acts on $f(x)$ by shifting the square roots of the complex variable $x$ by $\frac{i}{2}$.\ \ As a final remark, we note that this operator can actually be generalized to take other non-zero square-root shifts.

\begin{definition}
	Given each $c\in\mathbb{C}\setminus\{0\}$, we let $\sqrt{\cdot}$ be a branch of the complex square-root with the line joining $c$ and $0$ as the branch cut.\ \ Then for each $x\in\mathbb{C}$ we denote
	\[
		x^+ := \left(\sqrt{x}+\frac{c}{2}\right)^2 \hspace{30px} \mbox{and} \hspace{30px} x^- := \left(\sqrt{x}-\frac{c}{2}\right)^2,
	\]
	and define the \textbf{\textit{$\boldsymbol{c}$-shift Wilson operator}} $\D_{\W,c}$ by
	\[
		(\D_{\W,c} f)(x) := \frac{f(x^+)-f(x^-)}{x^+ - x^-} = \frac{f((\sqrt{x}+\frac{c}{2})^2)-f((\sqrt{x}-\frac{c}{2})^2)}{2c\sqrt{x}}.
	\]
\end{definition}

All the results for $\D_\W\equiv\D_{\W,i}$ in this paper still hold for $\D_{\W,c}$. \\

It is known that the Askey-Wilson operator $\D_q$ is defined for every $q\in D(0;1)\setminus\{0\}$, and reduces to the ordinary differential operator $\frac{d}{dx}$ as $q\to 1$ from the interior of $D(0;1)$.\ \ In fact, the $c$-shift Wilson operator $\D_{\W,c}$ behaves in a similar way.\ \ It can be easily shown that $\D_{\W,c}$ reduces to the usual differential operator as $c\to 0$.

\begin{proposition}
	Let $x\in\mathbb{C}$ and $f$ be a function holomorphic at $x$.\ \ Then
	\[
		\lim_{c\to 0}{(\D_{\W,c} f)(x)} = f'(x).
	\]
\end{proposition}

\bigskip
In this paper, we have developed a full-fledged Nevanlinna's value distribution theory of the Wilson divided-difference operator for finite order meromorphic functions.\ \ Key concepts and results including a lemma on logarithmic Wilson difference have been established for meromorphic functions of finite order.\ \ The finite-order restriction is generally the best possible. We have established a second fundamental theorem with respect to the Wilson operator, from which a new difference type little Picard theorem for meromorphic functions of finite order is derived.\ \ The Wilson-type Nevanlinna defect relations that follow from the analogous classical Nevanlinna formalism require a different way of counting zeros, poles and their multiplicities of  meromorphic functions, which is natural with respect to the Wilson operator.\ \ This new way of counting is the key to the new Wilson Nevanlinna theory. Halburd and Korhonen \cite{Halburd-Korhonen} were the first to observe new counting functions for the usual difference operator.\ \ The Nevanlinna theories for the Askey-Wilson operator \cite{Chiang-Feng3} as well as the Wilson operator established in this paper indicate that there are corresponding versions of residue calculus for these divided-difference operators, which may provide natural ways to better understand the corresponding special functions.\ \ These issues, as well as other function theoretic and interpolation-related investigations, will be discussed in subsequent papers.
\bigskip

\noindent \textbf{Acknowledgement.}\ \ We would like to thank the anonymous referee for his/her helpful and constructive comments and bringing Korhonen's work \cite{Korhonen} to our attention.\ \ We would also like to thank Dr. T. K. Lam for many valuable discussions throughout this research. \\

\bibliographystyle{amsplain}

\end{document}